\documentclass[11pt,a4paper]{article}
\usepackage[top=1in,bottom=1.1in,left=1in,right=1in]{geometry}
\usepackage{amssymb}
\usepackage{amsmath,color}
\usepackage{amsthm}
\usepackage{verbatim}
\usepackage{hyperref}
\usepackage[T1]{fontenc}
\usepackage[vlined]{algorithm2e}

\newtheorem{theorem}{Theorem}[section]
\newtheorem{proposition}[theorem]{Proposition}
\newtheorem{corollary}[theorem]{Corollary}
\newtheorem{lemma}[theorem]{Lemma}
\theoremstyle{definition}
\newtheorem{definition}[theorem]{Definition}
\newtheorem{remark}[theorem]{Remark}
\newtheorem{example}[theorem]{Example}
\newtheorem{cl}{Claim}
\newtheorem*{claim}{Claim}

\newcommand{\N}{\mathbb{N}}                                          
\newcommand{\R}{\mathbb{R}}                                                                         
\newcommand{\eps}{\varepsilon} 
\newcommand{\ds}{\displaystyle} 
\DeclareMathOperator{\CAT}{CAT}                                        
\DeclareMathOperator{\CBB}{CBB}   
\DeclareMathOperator{\Img}{Im} 
\DeclareMathOperator{\inte}{int} 
\DeclareMathOperator{\UAG}{UAG} 
\DeclareMathOperator{\dom}{dom}
\DeclareMathOperator{\epi}{epi}

\begin{document}

\title{Local linear convergence of alternating projections in metric spaces with bounded curvature}

\author{Adrian S. Lewis$^{a}$, Genaro L\'{o}pez-Acedo$^{b}$, Adriana Nicolae$^{c}$}
\date{}
\maketitle

\begin{center}
{\scriptsize
$^{a}$School of Operations Research and Information Engineering, Cornell University, Ithaca, NY
\ \\
$^{b}$Department of Mathematical Analysis - IMUS, University of Seville, C/ Tarfia s/n, 41012 Seville, Spain
\ \\
$^{c}$Department of Mathematics, Babe\c s-Bolyai University, Kog\u alniceanu 1, 400084 Cluj-Napoca, Romania \\
\ \\
E-mail addresses: adrian.lewis@cornell.edu (A. S. Lewis), glopez@us.es (G. L\'{o}pez-Acedo),\\ 
anicolae@math.ubbcluj.ro (A. Nicolae)
}
\end{center}

\maketitle

\begin{abstract}
We consider the popular and classical method of alternating projections for finding a point in the intersection of two closed sets.  By situating the algorithm in a metric space, equipped only with well-behaved geodesics and angles (in the sense of Alexandrov), we are able to highlight the two key geometric ingredients in a standard intuitive analysis of local linear convergence.  The first is a transversality-like condition on the intersection;  the second is a convexity-like condition on one set:  ``uniform approximation by geodesics''.\\

\noindent {\em MSC:  65K10 (primary),  53C20 (secondary)} 
\\

\noindent {\em Keywords}: Alternating projections, \and linear convergence, \and Alexandrov spaces, \and almost convexity, \and curvature bounds.
\end{abstract}

\section{Introduction}

An important problem with a large range of applications in many branches of mathematics and other sciences consists in finding a point in the intersection of two given closed sets $A$ and $B$ in $\R^n$. Under the assumption that the sets are additionally convex, a straightforward algorithm widely used since the work of von Neumann \cite{vNeu50} is the alternating projection method which defines a sequence $(x_n)$ by iteratively projecting the current point (that belongs to one of the sets) onto the other set. Alternating projections can be slow to converge, but in the presence of appropriate conditions regarding the way the two sets intersect, it is still possible to obtain linear convergence meaning that the distance from $x_n$ to the limit point is bounded above by $k\cdot a^n$, where $k$ and $a$ are positive constants with $a<1$. A classical assumption going back to Gubin, Polyak, and Raik \cite{GubPolRai67} says that one set should intersect the interior of the other one. Other less restrictive intersection conditions can be found in the survey by Bauschke and Borwein \cite{BauBor93}. 

When dealing with sets that are not necessarily convex, the metric projection might be multi-valued, and although one can adapt the definition of the method to choose one of the nearest points, in general one cannot expect to obtain global convergence to a solution of the problem. However, it was noticed in practice that locally, that is, if the initial point is close enough to the intersection of the two sets, the alternating projection method works well. This observation and the variety of important applications in absence of convexity triggered in recent years a closer analysis of this classical method which led to a series of local linear convergence results assuming appropriate conditions at an intersection point $z$ of the two sets and considering the starting point $x_0$ sufficiently close to $z$. 

Lewis, Luke, and Malick \cite{LewLukMal09} focused on two local geometric features needed to prove linear convergence of alternating projections in an elementary way. Namely, they showed that if $A$ and $B$ intersect transversally at $z \in A \cap B$ and one of the sets is super-regular at $z$, then the alternating projection method converges linearly to an intersection point (provided $x_0$ is chosen sufficiently close to $z$). Transversality is a standard geometric condition and requires that $N_A(z) \cap (-N_B(z))= \{0\}$, where $N_A(z)$ and $N_B(z)$ are the normal cones to $A$ and $B$, respectively, at $z$. For convex $A$ and $B$, transversality actually means that the sets cannot be separated by a hyperplane. Super-regularity generalizes a fundamental property of a closed convex set which characterizes the projection of a given point onto the set via the property that the point, its projection, and any point of the set form a right or obtuse angle. In this sense, super-regularity means in broad lines that the set is not too far from being convex. More precisely, $A$ is called super-regular at $z$ if given any $\eps > 0$, for any $y \notin A$ and $x' \in A$ sufficiently close to $z$ and any $x \in P_A(y)$ with $x \ne x'$, we have that the angle between the vectors $y-x$ and $x'-x$ is at least $\pi/2- \eps$ (here $P_A$ denotes the metric projection onto the set $A$). For a closed set, convexity clearly implies super-regularity. In the nonconvex case, there are also other well-known properties that are highly relevant in practice and imply super-regularity such as prox-regularity. The convergence result from \cite{LewLukMal09} was later improved by Drusvyatskiy, Ioffe, and Lewis \cite{DruIofLew15} by removing the hypothesis that one of the sets has to be super-regular. In fact, they even weakened the transversality condition by replacing it with another geometric property called intrinsic transversality, which, as opposed to transversality, is a property specific to the sets and is not related to the Euclidean space where they lie. 

A notion related to intrinsic transversality was introduced by Noll and Rondepierre \cite{NolRen16} and called $0$-separability. This notion combined with H\"{o}lder regularity is weaker than intrinsic transversality and still yields local linear convergence of alternating projections. The set $A$ intersects the set $B$ separably at $z\in A \cap B$ if there exists $\alpha > 0$ such that, for any $x \in A \setminus B$ sufficiently close to $z$ and any $y \in P_B(x) \setminus A$ and $x' \in P_A(y)$, the angle between the vectors $x-y$ and $x'-y$ is at least $\alpha$. Actually, as noted by Drusvyatskiy and Lewis \cite{DruLew19}, for the intuitive geometric reasoning from \cite{LewLukMal09} to work, it is enough to assume super-regularity of one of the sets and, instead of transversality, its separable intersection with the other one.

The projection onto a set is in fact a purely metric concept, and so it is natural to consider the described problem in nonlinear settings with a rich enough geometry. Even in the convex case in Hilbert spaces, both the orthogonality of the metric projection and the geometric structure which allows one to establish a relation between angles and sides of triangles seem essential to this approach to convergence of alternating projections. Since the structural characteristics of the distance function in Alexandrov spaces with an upper curvature bound  are sufficient for the aforementioned properties to hold, we consider this context and briefly discuss in Section \ref{sect-prelim} some of its basic properties, along with other notions used in what follows.

As a first step in this direction  we introduce in Section \ref{sect-curv-conv} the notion of uniform approximation by geodesics at a point for subsets of a geodesic space, a property that is satisfied, in Alexandrov spaces with an upper curvature bound, by locally compact sets with finite extrinsic curvature at that point (see Definition \ref{defi-fi-cur} and Proposition \ref{prop-subset-extc}). We relate uniform approximation by geodesics with other concepts that indicate how close a set is to being convex such as the notions of $2$-convexity, an important case of almost convexity introduced by Lytchak \cite{Lyt05}, positive reach in the sense of Federer \cite{Fed59}, or prox-regularity originating from Poliquin and Rockafellar \cite{PolRoc96}. Moreover, we show that transforming a convex set via a sufficiently smooth function defined between two Euclidean spaces results in a set that is uniformly approximable by geodesics. Furthermore, we justify why the notion of uniform approximation by geodesics differs from the convexity-like notions mentioned before. Other natural classes of sets that are uniformly approximable by geodesics include epigraphs of approximately convex functions or sets defined by $C^1$ inequality constraints satisfying the Mangasarian--Fromovitz condition.

In Section \ref{sect-alt-proj} we study the two main ingredients, super-regularity and separable intersection, used to prove the local linear convergence of alternating projections. In Theorem \ref{thm-prop-p-superreg}, we prove that in Alexandrov spaces with an upper curvature bound, uniform approximation by geodesics implies super-regularity. Using the analysis carried out in Section 3, this allows us to conclude that in suitable settings (in particular, in smooth Riemannian manifolds) compact sets that are $2$-convex as well as images of convex sets under sufficiently smooth functions between two Euclidean spaces are super-regular. Even though our main focus is the study of nonconvex sets, for the separable intersection property we also consider the convex case. To this end we introduce a nonlinear counterpart of transversality of two sets at a point and show that it implies separable intersection at that point. Besides, we prove that transversality holds when assuming that one set intersects the interior of the other set, thus generalizing the condition from \cite{GubPolRai67}  to the framework of Alexandrov spaces. 
Finally, we extend the local linear convergence of alternating projections to the setting of Alexandrov spaces of curvature bounded above. The study of the alternating projection method for convex sets in Alexandrov spaces was initiated by Ba\v{c}\'{a}k, Searston, and Sims \cite{BacSeaSim12} who showed that in the presence of nonpositive curvature the algorithm works well (see also \cite{ChoJiLim18} for the case of an arbitrary upper curvature bound). Regarding the nonconvex case in the context of Alexandrov spaces, as far as we know, Theorem \ref{thm-main-conv} is the first convergence result in this line.

\section{Preliminaries} \label{sect-prelim}
This section introduces some notation and briefly explains certain properties of geodesic metric spaces that we make use of in what follows. These concepts and related ones are treated at length, e.g., in the monographs \cite{AleKapPet19,Bri99,Bur01}.

\subsection{Geodesic spaces}
Let $(X,d)$ be a metric space. We denote the open (resp., closed) ball centered at $x\in X$ with radius $r>0$ by $B(x,r)$ (resp., $\overline{B}(x,r)$). For $C \subseteq X$, the {\it metric projection} $P_C$ onto $C$ is the mapping $P_C : X\to 2^C$ defined by
\[P_C(z)=\{ y \in C : d(z,y)=\mbox{dist}(z,C)\}, \quad z\in X,\]
where  $\mbox{dist}(z,C) = \inf_{y \in C}d(z,y)$.

Let $x,y \in X$. A {\it curve} from $x$ to $y$ is a continuous mapping $f : [a,b] \subseteq \R \to X$ such that $f(a)=x$ and $f(b)=y$. A partition of $[a,b]$ is a finite and ordered set $\Delta = (t_0, t_1, \ldots, t_n)$ of points in $[a,b]$ such that $a=t_0 \le t_1 \le \ldots \le t_n = b$. The {\it length} of a curve $f : [a,b] \to X$ is $\sup_{\Delta}\sum_{i=0}^{n-1}d(f(t_i),f(t_{i+1}))$, where the supremum is taken over all partitions $\Delta$ of $[a,b]$. If the length of $f$ is finite, then $f$ is called {\it rectifiable}. If the length of the restriction $f|_{[t_1,t_2]}$ equals $t_2 - t_1$ for all $t_1,t_2 \in [a,b]$ with $t_1 \le t_2$, then $f$ is {\it parameterized by arc length} (or {\it unit-speed}). Note that every rectifiable curve can be reparameterized by arc length.

Given a metric space $(X,d)$, define $d^X : X \times X \to [0,\infty]$ by assigning to a pair of points $x, y \in X$ the infimum of the lengths of rectifiable curves from $x$ to $y$ (if there are no such curves, then $d^X (x,y) = \infty$). If every two points in $X$ can be joined by a rectifiable curve, $d^X$ is a metric called the {\it length metric} associated to $d$. If $d=d^X$, then $(X,d)$ is called a {\it length space}. Note that $(X,d^X)$ is a length space. The {\it induced length metric}  $d^A$ on $A \subseteq X$ is the length metric associated to the restriction of $d$ to $A \times A$. 

If $X$ is a proper metric space and $x, y \in X$ such that there exists a rectifiable curve in $X$ joining $x$ and $y$, then there exists a curve whose length is equal to the infimum of the lengths of curves joining $x$ and $y$ (see, e.g., \cite[Proposition 1.4.12]{Pap05}).

Let $x,y \in X$. A {\it geodesic} from $x$ to $y$ is a mapping $\gamma:[0,l]\subseteq\R\to X$ such that $\gamma(0)=x$, $\gamma(l)=y$, and 
\[d(\gamma(t_1),\gamma(t_2))=|t_1-t_2|\quad \text{for all }t_1,t_2\in[0,l].\]
In this case we also say that $\gamma$ {\it joins} $x$ and $y$. The image of a geodesic joining $x$ and $y$ is called a {\it geodesic segment} joining $x$ and $y$. Note that $d(x,y) = l$, which is also equal to the length of $\gamma$. 

We say that $(X,d)$ is a {\it (uniquely) geodesic space} if every two points in $X$ are joined by a (unique) geodesic segment. Every geodesic space is a length space. Conversely, by the Hopf--Rinow theorem, every length space that is complete and locally compact is geodesic.

Let $(X,d)$ be a metric space, and consider the direct product $X \times \R$
\index{product of metric spaces}
equipped with the metric
\[d_2((x_1,y_1),(x_2,y_2)) = \sqrt{d(x_1,x_2)^2 + |y_1 - y_2|^2},\]
where $x_1, x_2 \in X$ and $y_1, y_2 \in \R$. If $X$ is a geodesic space, then so is $X \times \R$. Moreover, if $\gamma : [0,r] \to X$ and $\alpha : [0,s] \to \R$ are geodesics, where $l = \sqrt{r^2 + s^2} > 0$, then $\sigma : [0,l] \to X \times \R$ defined by $\sigma(t) = (\gamma(rt/l), \alpha(st/l))$ for all $t \in [0,l]$ is a geodesic. 

If $(X,d)$ is a uniquely geodesic space, $z\in X$, and $c > 0$, we say that $X$ has the {\it $c$-geodesic extension property} at $z$ if for any distinct $x,y \in B(z,c)$ with $d(x,y) < c$, the geodesic from $x$ to $y$ can be extended beyond $y$ to a geodesic of length $c$.  

In the following, if nothing else is mentioned, we always assume that $(X,d)$ is a geodesic space. A point $z \in X$ belongs to a geodesic segment joining $x$ and $y$ if and only if there exists $t\in [0,1]$ such that $d(x,z)=td(x,y)$ and $d(y,z)=(1-t)d(x,y)$. If there is a unique geodesic segment joining $x$ and $y$, we denote it by $[x,y]$. A subset $A$ of $X$ is {\it convex} if for every $x,y \in A$, all geodesic segments that join $x$ and $y$ are contained in $A$ and {\it weakly convex} if for every $x,y \in A$, at least one geodesic segment that joins $x$ and $y$ is contained in $A$. 

Given $\lambda \in \R$, a function $f : I \to \R$ defined on an interval $I \subseteq \R$ is called $\lambda$-convex if $f(t) + \lambda t^2$ is convex on $I$. More generally, we say that a function $f : X \to \R$ is {\it $\lambda$-convex} if for every geodesic $\gamma : [0,l] \to X$, $(f \circ \gamma)(t) + \lambda t^2$ is convex on $[0,l]$. This amounts to the inequality
\[f(\gamma((1-\alpha)t_1 + \alpha t_2)) \le (1-\alpha)f(\gamma(t_1)) + \alpha f(\gamma(t_2)) + \lambda \alpha(1-\alpha)d(\gamma(t_1),\gamma(t_2))^2\]
for all $t_1, t_2 \in [0,l]$ and all $\alpha \in [0,1]$.

A {\it geodesic triangle} $\Delta(x_1,x_2,x_3)$ in $X$ consists of three points $x_1, x_2, x_3 \in X$ (its vertices) and three geodesic segments (its sides) joining each two points. A triangle $\Delta(\overline{x}_1, \overline{x}_2, \overline{x}_3)$ in $\R^2$ is a {\it comparison triangle} for a geodesic triangle $\Delta(x_1,x_2,x_3)$ if $d(x_i,x_j) = d_{\R^2}(\overline{x}_i,\overline{x}_j)$ for $i,j \in \{1,2,3\}$. 

Let $\gamma:[0,l]\to X$ and $\gamma':[0,l']\to X$ be two nonconstant geodesics with $\gamma(0)=\gamma'(0)$. For $t \in (0,l]$ and $t' \in (0,l']$, consider a comparison triangle $\Delta(\overline{\gamma(0)}, \overline{\gamma(t)}, \overline{\gamma'(t')})$ in $\mathbb{R}^2$, and denote its interior angle at $\overline{\gamma(0)}$ by $\overline{\angle}_{\gamma(0)}\left(\gamma(t),\gamma'(t')\right)$. The {\it Alexandrov angle} $\angle(\gamma,\gamma')$ between the geodesics $\gamma$ and $\gamma'$ is defined as 
\[\angle(\gamma,\gamma') = \limsup_{t,t' \to 0}\overline{\angle}_{\gamma(0)}\left(\gamma(t),\gamma'(t')\right) \in [0,\pi].\] 

For $x,y,z$ three points in a geodesic space $X$ with $x \ne y$ and $x \ne z$, if there is a unique geodesic from $x$ to $y$, as well as from $x$ to $z$, then we denote the corresponding Alexandrov angle by $\angle_x(y,z)$. In general, the sum of adjacent angles is at least $\pi$. In other words, suppose that $\gamma:[0,l]\to X$ is a nonconstant geodesic, and let $t_0 \in (0,l)$. Define $\gamma_1:[0,t_0]\to X$ by $\gamma_1(t)=\gamma(t_0-t)$ and $\gamma_2:[0,l-t_0]\to X$ by $\gamma_2(t)=\gamma(t_0+t)$. If $\gamma_3$ is a nonconstant geodesic satisfying $\gamma_3(0)=\gamma(t_0)$, then $\angle(\gamma_1,\gamma_3) + \angle(\gamma_3,\gamma_2) \ge \pi$. 

\subsection{Alexandrov spaces}

For $\kappa \in \R$, let $M_\kappa^2$ denote the complete, simply connected, $2$-dimensional Riemannian manifold of constant sectional curvature $\kappa$. In particular,  $M_0^2 = \R^2$. Recall that, for $\kappa > 0$, $M_\kappa^2$ is obtained from the spherical space $\mathbb{S}^2$ by scaling the spherical distance with $1/\sqrt{\kappa}$, while for $\kappa < 0$, $M_\kappa^2$ is obtained from the hyperbolic space $\mathbb{H}^2$ by scaling the hyperbolic distance with $1/\sqrt{-\kappa}$. We denote the diameter of $M_\kappa^2$ by $D_\kappa$. In other words, $D_\kappa = \infty$ if $\kappa \le 0$, while $D_\kappa = \pi/\sqrt{\kappa}$ if $\kappa > 0$. 

The cosine law in $M_\kappa^2$ states that in a geodesic triangle with vertices $x,y,z$ and vertex angle $\alpha$ at $x$ (which is equal to the Alexandrov angle determined by the side joining $x$ and $y$ and the one joining $x$ and $z$; see \cite[Chapter I, Proposition 2.9]{Bri99}) we have for $\kappa > 0$,  
\[\cos(\sqrt{\kappa}d(y,z)) = \cos(\sqrt{\kappa}d(x,y))\cos(\sqrt{\kappa}d(x,z))+ \sin(\sqrt{\kappa}d(x,y))\sin(\sqrt{\kappa}d(x,z))\cos \alpha.\]

The subsequent result  is a consequence of the cosine law and will be used in the proof of Theorem \ref{thm-prop-p-superreg}. Although it is straightforward, we include the proof for completeness.

\begin{lemma}\label{lemma-small-angle}
Let $\kappa > 0$, $x,y,z \in M_\kappa^2$ with $d(y,x) < D_\kappa / 2$, $d(z,x) < D_\kappa$ and $\angle_{x}(y,z) < \pi/2$. Then for all $t > 0$ there exists $u \in [x,z]$ with $d(x,u) < t$ and $d(y,u) < d(y,x)$.
\end{lemma}
\begin{proof}
Let $w = P_{[x,z]}(y)$. Then $\alpha = \angle_{w}(y,x) \ge \pi/2$. Note that $w \ne x$ since $\angle_{x}(y,z) < \pi/2$.  Let $u \in [x,w]$ with $u \ne x$. Applying the cosine law we get
\[\cos(\sqrt{\kappa}\,d(y,x)) = \cos(\sqrt{\kappa}\,d(y,w)) \cos(\sqrt{\kappa}\,d(x,w))  + \sin(\sqrt{\kappa}\,d(y,w)) \sin(\sqrt{\kappa}\,d(x,w)) \cos \alpha\]
and 
\[\cos(\sqrt{\kappa}\,d(y,u)) = \cos(\sqrt{\kappa}\,d(y,w)) \cos(\sqrt{\kappa}\,d(u,w))  + \sin(\sqrt{\kappa}\,d(y,w)) \sin(\sqrt{\kappa}\,d(u,w)) \cos \alpha.\]
Since $\cos \alpha \le 0$ and $d(u,w) < d(x,w)$, it follows that $d(y,u) < d(y,x)$.
\end{proof}

Let $(X,d)$ be a geodesic space. As in $\R^2$, one can define comparison triangles in $M_\kappa^2$. A triangle $\Delta(\overline{x}_1, \overline{x}_2, \overline{x}_3)$ in $M^2_{\kappa}$ is a {\it comparison triangle} for a geodesic triangle $\Delta(x_1,x_2,x_3)$ in $X$ if $d(x_i,x_j) = d_{M^2_{\kappa}}(\overline{x}_i,\overline{x}_j)$ for $i,j \in \{1,2,3\}$. For $\kappa$ fixed, comparison triangles in $M^2_{\kappa}$ of geodesic triangles having perimeter less than $2D_\kappa$ always exist and are unique up to isometry.

A geodesic triangle $\Delta=\Delta(x_1,x_2,x_3)$ is said to satisfy the {\it $\CAT(\kappa)$ inequality} if for every comparison triangle $\overline{\Delta} = \overline{\Delta}(\overline{x}_1,\overline{x}_2,\overline{x}_3)$ in $M^2_\kappa$ of $\Delta$ and every $x,y \in \Delta$ we have
\[d(x,y) \le d_{M^2_{\kappa}}(\overline{x},\overline{y}),\]
where  $\overline{x},\overline{y} \in \overline{\Delta}$ are the comparison points of $x$ and $y$; i.e., if $x$ belongs to the side joining $x_i$ and $x_j$, then $\overline{x}$ belongs to the side joining $\overline{x}_i$ and $\overline{x}_j$ and satisfies $d(x_i,x)=d_{M_\kappa^2}(\overline{x}_i,\overline{x})$.

We include in what follows some definitions and properties concerning metric spaces that have globally upper or lower curvature bounds in the sense of Alexandrov.

Let $X$ be a metric space where every two points at distance less than $D_\kappa$ can be joined by a geodesic. We say that $X$ is a {\it $\CAT(\kappa)$ space} (or that it has {\it curvature bounded above by $\kappa$} in the sense of Alexandrov) if every geodesic triangle having perimeter less than $2D_\kappa$ satisfies the $\CAT(\kappa)$ inequality. Another equivalent condition for $X$ to be a $\CAT(\kappa)$ space is that the Alexandrov angle between the sides of any geodesic triangle in $X$ (of perimeter smaller than $2D_\kappa$) is less than or equal to the angle between the corresponding sides of its comparison triangle in $M_\kappa^2$.  A $\CAT(\kappa)$ space is also a $\CAT(\kappa')$ space for every $\kappa'\geq \kappa$. Other characterizations of $\CAT(\kappa)$ spaces are given, e.g., in \cite[Chapter II, Proposition 1.7 and Theorem 1.12]{Bri99}.

We briefly mention next some facts on $\CAT(\kappa)$ spaces needed further on. Let $X$ be a $\CAT(\kappa)$ space. Observe first that points in $X$ at distance less than $D_\kappa$ are joined by a unique geodesic segment and balls of radius smaller than $D_\kappa/2$ are convex.

Assuming that the diameter of $X$ is sufficiently small for $\kappa > 0$, there exists $\lambda < 1$ such that, for all $x \in X$, the function $d(\cdot, x)^2$ is $(-\lambda)$-convex. In other words, $\CAT(\kappa)$ spaces are $2$-uniformly convex (see \cite{Kuw14, Oht07}). Moreover, there also exists $\lambda'>0$ such that, for all geodesic segments $[x,y]$ in $X$, the function $\text{dist}(\cdot, [x,y])$ is $\lambda'$-convex (see, e.g., \cite[Lemma 8.6.6]{AleKapPet19}). Note that, by reducing the diameter of the space, the convexity constants $\lambda$ and $\lambda'$ can be made arbitrarily close to $1$ and arbitrarily close to $0$, respectively. 

If $x,y,z \in X$ with $\max\{d(x,y),d(x,z)\} < D_\kappa$, then the function $(x,y,z) \mapsto \angle_x(y,z)$ is upper semicontinuous, and if $x$ is fixed, then the function $(y,z) \mapsto \angle_x(y,z)$ is continuous (see, e.g., \cite[Chapter II, Proposition 3.3]{Bri99}). Note also that one can take comparison triangles in $M_\kappa^2$ instead of $\R^2$ in the definition of Alexandrov angles. Thus, if $\angle_x(y,z) < \alpha$ for some $\alpha > 0$, then, using the definition of Alexandrov angles, there exist $y' \in [x,y]$ and $z' \in [x,z]$ such that in a comparison triangle $\overline{\Delta}(\overline{x},\overline{y}', \overline{z}')\subseteq M_\kappa^2$ of the geodesic triangle $\Delta(x,y',z')$, the angle at $\overline{x}$ is smaller than $\alpha$.

Suppose now that $X$ is additionally complete, and let $C \subseteq X$ be nonempty, closed, and convex. Take $x \in X$ with $\text{dist}(x,C) < D_\kappa/2$. Then $P_C(x)$ is a singleton. Moreover, if $x \notin C$, $y \in C$ with $y \ne P_C(x)$, then $\angle_{P_C(x)}(x,y) \ge \pi/2$. Furthermore, we will also use the fact that, for all $y \in C$ with $d(P_C(x),y) \le D_\kappa/2$, we have $d(P_C(x),y) \le d(x,y)$ (see, e.g., \cite[Proposition 3.5]{EspFer09}). A systematic study of properties of Chebyshev sets (i.e., sets where the metric projection is a singleton for all points in the space) in Alexandrov spaces is carried out in \cite{AriFer16}.

{\it $\CBB(\kappa)$ spaces} (or spaces that have {\it curvature bounded below by $\kappa$} in the sense of Alexandrov) are defined in a similar way to $\CAT(\kappa)$ spaces using in this case the reverse of the $\CAT(\kappa)$ inequality. As for $\CAT(\kappa)$ spaces, one can give other equivalent conditions involving, e.g., angles. Note also that in $\CBB(\kappa)$ spaces, the sum of adjacent angles is precisely equal to $\pi$. A $\CBB(\kappa)$ space is also a $\CBB(\kappa')$ space for all $\kappa' \le \kappa$.

Given $\kappa' \leq \kappa$, following the literature (see, e.g., \cite[Chapter II.9]{BN93}), we will refer to a geodesic space that is both a $\CAT(\kappa)$ space and a $\CBB(\kappa')$ space as an $\Re_{\kappa', \kappa}$ domain. The behavior of distance functions in spaces with local two-sided curvature bounds has been very recently discussed in \cite{KapLyt21}.

Curvature bounds in the sense of Alexandrov can also be considered locally by imposing the comparison condition only for triangles that are small enough. More precisely, we say that a metric space has locally curvature bounded above (resp., below) by $\kappa$ if every point in it has a neighborhood that, with the induced metric, is a $\CAT(\kappa)$ space (resp., a $\CBB(\kappa)$ space). 

If $z$ is a point in a smooth Riemannian manifold, then there exists a neighborhood of $z$ that is an $\Re_{\kappa',\kappa}$ domain for some suitable $\kappa, \kappa' \in \R$ (see, e.g., \cite[Chapter II.1, Appendix]{Bri99} and \cite{BurGroPer92}). Furthermore, bearing in mind that the injectivity radius in  a complete  Riemannian manifold is a continuous function \cite[Part I, Proposition 2.1.10]{Kli82}, there exists a ball centered at $z$ that is an $\Re_{\kappa', \kappa}$ domain with the $c$-geodesic extension property at $z$ for an appropriate $c > 0$.

The structure of $\CAT(\kappa)$ spaces satisfying a local geodesic extension property has been analyzed in \cite{LytNag19}.

\noindent {\bf Notation.} $\N$ denotes the set of natural numbers including $0$.

\section{Curvature and convexity} \label{sect-curv-conv}

As pointed out by Gromov in \cite{Gro91}, the full geometric meaning of the curvature tensor is obscure, which motivates the search for a suitable notion of curvature for a subspace of a metric space. In the case of a curve, a first approach  was made by Menger who mimicked the notion of classical curvature for plane curves. However, this definition is limited because it was  modeled after the Euclidean plane. Our approach here is the one suggested by Finsler in his Ph.D. thesis and studied by Haantjes \cite{Han47} using the comparison between the length of an arc and the distance between its endpoints. A nice introduction of different concepts of metric curvature for curves is given in \cite{Sau15}.
  
\begin{definition}\label{defi-fi-cur}
Let $(X,d)$ be a metric space, $A\subseteq X$, and $z\in A$. We say that $A$ has {\it finite extrinsic curvature} at $z$ if there exist $\sigma \ge 0$ and $r > 0$ such that
\begin{equation}\label{def-extc-eq}
d^A(p,q) - d(p,q) \le \sigma \, d(p,q)^3
\end{equation}
for all $p, q \in B(z,r) \cap A$. 
\end{definition}

The above definition is a local version of the notion of almost convexity (more precisely, $2$-convexity) introduced by Lytchak  \cite{Lyt05} in the following way: given $\sigma \ge 0$ and $\rho > 0$, a subset $A$ of $X$ is called $(\sigma,2,\rho)$-convex (or $2$-convex when there is no need to emphasize the constants $\sigma$ and $\rho$) if \eqref{def-extc-eq} holds for all $p,q \in A$ with $d(p,q) \le \rho$. It is immediate that a $(\sigma,2,\rho)$-convex set has finite extrinsic curvature at all its points. Actually, in \cite{AleBis06}, the terminology subspaces of extrinsic curvature bounded above is used for $2$-convex sets. It is also immediate that if $A$ is weakly convex, then it is $(0,2,\rho)$-convex for all $\rho > 0$ (in other words, $A$ is of extrinsic curvature $0$).

\begin{remark} \label{rmk-extc-2-convex}
If $X$ is a $\CAT(\kappa)$ space and $A \subseteq X$ is locally compact (in the subspace topology) and has finite extrinsic curvature at $z \in A$, then for all sufficiently small $R > 0$, $B(z,R) \cap A$ is $2$-convex.

Indeed, let $\sigma \ge 0$ and $r > 0$ be such that $\eqref{def-extc-eq}$ holds for all $p,q \in B(z,r) \cap A$. We can assume that $r$ is small enough so that, on $B(z,r)$, the squared distance function to a point is $(-1/2)$-convex and any two points in $B(z,r) \cap A$ are joined by a geodesic in $(A, d^A)$ (for this one uses the local compactness of $A$ and the finite extrinsic curvature at $z$). Take 
\[\alpha = \min\left\{r,\frac{1}{32\sqrt{\sigma} + 1}\right\} \quad \text{and} \quad R = \frac{\alpha}{3(\sigma + 1)}.\] 

Let $p,q \in B(z,R) \cap A$, and denote $s = d^A(p,q)$. By \eqref{def-extc-eq},
\[s \le d(p,q) + \sigma d(p,q)^3 < 2R(\sigma+1).\]
Let $f : [0,s] \to A$ be a geodesic in $(A, d^A)$ with $f(0) = p$ and $f(s) = q$. If $t \in [0,s]$, then
\[d(z,f(t)) \le d(z,p) + d(p,f(t)) < R + d^A(p,f(t)) = R + t \le R + s < R + 2R(\sigma+1) < \alpha,\]
so $f(t) \in B(z,\alpha) \cap A$.

The function $g : B(z,\alpha)  \to \R$ defined by $g(x) = \frac{1}{2}d(x,z)^2$ is $\alpha$-Lipschitz as
\begin{align*}
\left|g(x) - g(y)\right|& = \frac{1}{2} \cdot \left(d(x,z) + d(y,z)\right)\cdot \left|d(x,z) - d(y,z)\right| \le \alpha d(x,y)
\end{align*}
for all $x,y \in B(z,\alpha)$. 

Let $t_1, t_2 \in [0,s]$, and denote $m = \frac{1}{2}f(t_1) + \frac{1}{2}f(t_2)$ and $m' = f(\frac{t_1 + t_2}{2})$. We have
\begin{align*}
d(f(t_1),m') \le d^A(f(t_1),m') = \frac{1}{2}d^A(f(t_1), f(t_2)) \le \frac{1}{2}d(f(t_1),f(t_2))\left(1 + \sigma d(f(t_1),f(t_2))^2\right).
\end{align*}
Likewise, $d(f(t_2),m') \le \frac{1}{2}d(f(t_1),f(t_2))(1 + \sigma d(f(t_1),f(t_2))^2)$. Assuming $\alpha$ is sufficiently small, we get $d(m,m') \le 2\sqrt{\sigma}d(f(t_1),f(t_2))^2$.

Since $g$ is $(-1/4)$-convex, we have 
\begin{align*}
g(m') & \le g(m) + \alpha d(m,m') \\
& \le \frac{1}{2}(g \circ f)(t_1) + \frac{1}{2}(g \circ f)(t_2) - \frac{1}{16}d(f(t_1),f(t_2))^2 + 2\alpha\sqrt{\sigma}d(f(t_1),f(t_2))^2.
\end{align*}

As $\alpha \le 1/(32\sqrt{\sigma} + 1)$, it follows that $g \circ f$ is convex. This shows that $\Img f \subseteq B(z,R) \cap A$, which implies that $B(z,R) \cap A$ is $(\sigma,2,2R)$-convex.

Note that intersections of $A$ with sufficiently small closed balls are $2$-convex as well.
\end{remark}

We introduce next the notion of uniform approximation by geodesics which will be one of our main tools to analyze super-regularity in the setting of geodesic spaces.

\begin{definition} 
Let $(X,d)$ be a metric space and $A\subseteq X$. We say that the set $A$ is  {\it  uniformly approximable  by geodesics $(\UAG)$}  at $z \in A$ if for all $\eps > 0$ there exists $R > 0$ such that for every $x, x' \in B(z,R) \cap A$ with $x \ne x'$ there exist $f : [0,l] \to A$, with $f(0) = x$ and $f(l) = x'$, and a geodesic $\gamma : [0,l] \to X$ starting at $x$ with the property that
\begin{equation}\label{prop-p-diff-eq}
\frac{d(\gamma(t),f(t))}{t} < \eps \quad \text{for all } t \in (0,l].
\end{equation}

\end{definition}

Every weakly convex subset of a geodesic space is $\UAG$ at all its points. However, convexity is not essential for $\UAG$ to hold. We prove next that if a set has finite extrinsic curvature at a given point, then it is $\UAG$ at that point. 

\begin{proposition}\label{prop-subset-extc}
Let $X$ be a $\CAT(\kappa)$ space and $A \subseteq X$ be locally compact (in the subspace topology). If $A$ has finite extrinsic curvature at $z \in A$, then $A$ is $\UAG$ at $z$.
\end{proposition}

\begin{proof} 
This result improves a previous one from an earlier version of the paper. This improved statement and its proof are based on ideas communicated to us by Alexander Lytchak.

Let $\eps \in (0,1)$. From the proof of Remark \ref{rmk-extc-2-convex} we know that we can take $R > 0$ small enough so that each two points in $B(z,R) \cap A$ can be joined by a geodesic in $(A,d^A)$ with its image contained in $B(z,R) \cap A$ and $B(z,R) \cap A$ is $(\sigma, 2, 2R)$-convex for some $\sigma \ge 0$. Moreover, we can assume that
\[R \le \min\left\{\frac{1}{\sigma + 1}, \frac{\eps}{64(\sqrt{\sigma}+1)}\right\}\]
and that in $B(z,R)$, the distance function to geodesic segments is $1$-convex.

Let $x, x' \in B(z,R) \cap A$ with $x \ne x'$. Denote $l = d(x,x')$ and $s = d^A(x,x')$. Consider $\gamma : [0,l] \to X$ the geodesic from $x$ to $x'$ and a constant-speed geodesic $f : [0,l] \to B(z,R) \cap A$ in $(A,d^A)$ from $x$ to $x'$ satisfying 
\[d^A(f(t_1),f(t_2)) = \frac{s}{l}|t_1 - t_2|\] 
for all $t_1, t_2 \in [0,l]$.

The function $g : B(z, R)  \to \R$ defined by $g(y) = \text{dist}(y,[x,x'])$ is $1$-Lipschitz and $1$-convex. Denote $h = g \circ f$. As in the proof of Remark \ref{rmk-extc-2-convex}, one obtains that $h$ is $2(8\sqrt{\sigma} + 1)$-convex.

\begin{claim} \label{prop-subset-extc-claim}
For all $t \in [0,l]$, $h(t) \le 4R(8\sqrt{\sigma}+1)t$.
\end{claim}
\begin{proof}[Proof of Claim]
Note that $h(0)=h(l)=0$. Denote $C = 2(8\sqrt{\sigma}+1)$. One can easily verify by induction that
\begin{equation} \label{prop-subset-extc-eq1}
h\left(\frac{l}{2^n}\right) \le l C \left(1 - \frac{1}{2^n}\right)\frac{l}{2^n}
\end{equation}
for all $n \in \N$.  

Let $t \in (0,l]$. Then $t = (1-\alpha)\frac{l}{2^{n+1}} + \alpha \frac{l}{2^n}$ for some $n \in \N$ and $\alpha \in (0,1]$. As $h$ is $C$-convex,
\begin{align*}
h(t) & \le (1-\alpha)h\left(\frac{l}{2^{n+1}}\right) + \alpha h\left(\frac{l}{2^n}\right) + C \alpha (1-\alpha) \left(\frac{l}{2^{n+1}}\right)^2\\
& \le l C\left(t - (1-\alpha)\frac{l}{2^{2n+2}} - \alpha \frac{l}{2^{2n}} + \alpha (1-\alpha)\frac{l}{2^{2n+2}}\right) \quad \text{by } \eqref{prop-subset-extc-eq1}\\
& \le l C t \le 2RCt.
\end{align*}
\end{proof}

Let $t \in (0,l]$, and denote $p(t) = P_{[x,x']}(f(t))$. Then $d(f(t),p(t)) = h(t)$.

Suppose $\gamma(t) \in [x,p(t)]$. Then
\begin{align*}
d(p(t), \gamma(t)) &= d(x,p(t)) - t \le d(x,f(t)) + d(f(t),p(t)) - t \le d^A(x,f(t)) + d(f(t),p(t)) -t \\
& = \left(\frac{s}{l} - 1\right)t + d(f(t),p(t)) \le \sigma l^2 t + d(f(t),p(t)) \le 4\sigma R^2t + h(t).
\end{align*}

Suppose now $\gamma(t) \in [p(t),x']$. Then
\begin{align*}
d(p(t), \gamma(t)) &= t - d(x,p(t)) \le t - d(x,f(t)) + d(f(t),p(t)).
\end{align*}
Since 
\[t \le  d^A(x,f(t)) \le d(x,f(t))\left(1 + \sigma d(x,f(t))^2\right) \le d(x,f(t))\left(1 + 4\sigma R^2\right),\]
it follows that 
\begin{align*}
d(p(t), \gamma(t)) & \le \left(1 - \frac{1}{1 + 4\sigma R^2}\right)t + d(f(t),p(t)) \le 4\sigma R^2 t + h(t).
\end{align*}
Thus, in both cases,
\begin{align*}
d(f(t), \gamma(t)) &\le d(f(t),p(t)) + d(p(t), \gamma(t)) \le 4\sigma R^2 t + 2h(t)\\
&  \le 4R(\sigma R + 2(8\sqrt{\sigma} + 1)) t < 64R(\sqrt{\sigma} + 1)t \le \eps t.
\end{align*}

\end{proof}

\begin{remark}
Note that the upper bound on the curvature in Proposition \ref{prop-subset-extc} is only needed locally around $z$, so, in particular, smooth Riemannian manifolds provide a suitable setting for the space $X$. By \cite[Theorem 1.2]{Lyt05}, if $A$ is a compact subset of a smooth Riemannian manifold and there exist a number $C > 0$ and a neighborhood of $z$ where every two points at distance $s$ can be joined in $A$ by an arc length parameterized $C^{1,1}$ curve (i.e., $C^1$ with a locally Lipschitz derivative) of length smaller than $C s$, whose $C^{1,1}$ norm is bounded by $C$, then $A$ has finite extrinsic curvature at all its points and hence is $\UAG$ at all its points.
\end{remark}

An interesting question in optimization and control theory is when the image of a ball under a smooth function is a convex set (see, e.g.,  \cite{Polyak01}). In what follows we address a question in this line from the point of view of uniform approximation by geodesics and study the image of a convex set under a sufficiently smooth function.

\begin{proposition}\label{prop-conv-img-AC}
Let $C \subseteq \R^n$ be convex and $u \in C$. Suppose that $F: \R^n \to\R^m $ is differentiable on a neighborhood of $u$, $DF$ is continuous at $u$, and $DF(u)$ is injective. Furthermore, assume that $F$ is a homeomorphism between $C$ and $F(C)$. Then $F(C)$ is $\UAG$ at $F(u)$.
\end{proposition}
\begin{proof}

Denote $z=F(u)$, $F= (F_1 ,\ldots, F_m)$, and let $U$ be a neighborhood of $u$ where $F$ is differentiable. As $DF(u)$ is injective, there exists $K>0$ such that $\| DF(u)(v)\| \geq K$ for all unit vectors $v \in \R^n$. Let $\eps > 0$, and take $r > 0$ such that, for all $ y\in B(u, r)  \subseteq U$,
\[\Vert DF(y)(v)\Vert \geq \frac{K}{2} \quad \text{for all unit vectors } v \in \R^n\]
and
\[\Vert \nabla F_i (y) -  \nabla F_i (u) \Vert < {K\eps \over 4 m} \quad \text{for all } i \in \{1,\ldots, m\}.\]

Let $R>0$ such that $B(z, R) \cap F(C) \subseteq F(B(u, r) \cap C)$. Fix $x, x' \in B(z, R) \cap F(C)$ with $x \ne x'$. Then $x= F(y), x' = F(y')$ for some $y, y' \in  B(u, r) \cap C$. Take
\[v = \frac{y'-y}{\Vert y'-y \Vert}, \quad w =DF(y)(v), \quad \text{and} \quad l=\|y'-y\|\cdot \|w\|.\]
Define $f : [0,l] \to F(C)$ and  a geodesic $\gamma : [0,l] \to \R^m$ by 
\[f(t) = F\left(y + t \frac{ v}{\Vert w\Vert}\right) \quad \text{and} \quad \gamma (t) = F(y) + t \frac{ w}{ \Vert w\Vert}.\]
Fix $t \in (0,l]$. Then there exist $c_1, \ldots , c_m \in B(u,r)$ such that 
\[F\left( y + t \frac{ v}{ \Vert w\Vert}\right)- F(y) = \frac{t}{\Vert w\Vert} \left( \langle\nabla F_1 (c_1 ), v\rangle, \ldots, \langle\nabla F_m (c_m ), v\rangle \right).\]
Consequently,
\begin{align*}
\frac{\Vert f(t) -\gamma(t) \Vert}{t}  &= \frac{1}{\Vert w\Vert}  \left\Vert \left(\langle\nabla F_1 (c_1)-\nabla F_1(y), v\rangle,\ldots , \langle\nabla F_m(c_m) -
\nabla F_m(y), v\rangle\right)\right\Vert \\
& \le \frac{1}{ \Vert w\Vert} \sum_{i=1}^m \left(\Vert \nabla F_i (c_i)-\nabla F_i (u) \Vert + \Vert \nabla F_i (u)-\nabla F_i (y)) \Vert \right)\\
& <  \frac{2}{ K}   m \left(\frac{K \eps}{ 4 m} + \frac{K \eps}{ 4 m}\right)= \eps. 
 \end{align*}
\end{proof}

\begin{corollary}\label{cor-conv-img-AC}
Let $C \subseteq \R^n$ be convex and $u \in C$. Suppose that $F: \R^n \to\R^m $ is differentiable on a neighborhood of $u$, $DF$ is continuous at $u$, and $DF(u)$ is injective. Then for any sufficiently small $r > 0$, $F(B(u,r) \cap C)$ is $\UAG$ at $F(u)$.
\end{corollary}
\begin{proof}
Denote $F= (F_1 ,\ldots, F_m)$, and let $U$ be a neighborhood of $u$ where $F$ is differentiable. As $DF(u)$ is injective, there exists $K>0$ such that $\|DF(u)(v)\| \geq K\|v\|$ for all $v \in \R^n$. From the continuity of $DF$ at $u$, there exists $r > 0$ such that, for all $y \in B(u,r) \subseteq U$, 
\[\Vert \nabla F_i (y) -  \nabla F_i (u) \Vert \le {K \over 2 m} \quad \text{for all } i \in \{1,\ldots, m\}.\]
Let $y,y' \in B(u,r)$. Then there exist $c_1, \ldots , c_m \in B(u,r)$ such that 
\[F(y')- F(y) = (\langle\nabla F_1 (c_1 ), y'-y\rangle, \ldots, \langle\nabla F_m (c_m ), y'-y\rangle).\]
Consequently,
\[ \|F(y') - F(y) - DF(u)(y'-y)\| \le \sum_{i = 1}^m \|\nabla F_i(c_i) - \nabla F_i(u)\|\cdot \|y'-y\| \le \frac{K}{2}\|y'-y\|,\]
from where 
\begin{equation}\label{cor-conv-img-AC-eq1}
\|F(y') - F(y)\| \ge \|DF(u)(y'-y)\| - \frac{K}{2}\|y'-y\| \ge \frac{K}{2}\|y'-y\|.
\end{equation}

Let $r' \le r$. By \eqref{cor-conv-img-AC-eq1}, $F$ is injective on $B(u,r')$. Moreover, $F^{-1} : F(B(u,r') \cap C) \to B(u,r') \cap C$ is continuous since, using \eqref{cor-conv-img-AC-eq1}, for all $x,x' \in F(B(u,r') \cap C)$, 
\[\|F^{-1}(x) - F^{-1}(x')\|  \le \frac{2}{K}\|x-x'\|.\]
Apply Proposition \ref{prop-conv-img-AC} for the convex set $B(u,r') \cap C$.
\end{proof}

We provide next an example of a set that is $\UAG$ at a point but does not have finite extrinsic curvature there.

\begin{example}
By Proposition \ref{prop-conv-img-AC} applied for the function $F : \R \to \R^2$, $F(x) = (x,|x|^{3/2})$, the set $A = \Img F$
is $\UAG$ at $(0,0)$.
 
For $\eps > 0$, the points $p_\eps= (-\varepsilon, \varepsilon^{3/2}), q_\eps = (\varepsilon, \varepsilon^{3/2}) \in \Img F$ satisfy
\[\lim_{\eps \to 0} {d^{A}(p_\eps,q_\eps) - d(p_\eps,q\eps) \over d(p_\eps,q_\eps)^3} = \infty;\] 
hence $A$ does not have finite extrinsic curvature at $(0,0)$.
\end{example}

Recently, much attention has been given to the epigraph of a certain class of functions that are close to being convex in the sense defined below. 
  
Let $X$ be a geodesic space and $f : X \to (-\infty,\infty]$. The (effective) {\it domain} of $f$ is the set $\dom f = \{x \in X \mid f(x) < \infty\}$, and the {\it epigraph} of $f$ is defined by 
\[\epi f = \{ (x,\lambda) \in \dom f \times \R \mid \lambda \ge f(x)\}.\] 

Following \cite{NLT00}, we say that $f$ is {\it approximately convex} at $z \in X$ if for all $\eps > 0$ there exists $r > 0$ such that for all $x,x' \in B(z,r)$, any geodesic $\gamma : [0,l] \to X$ from $x$ to $x'$, and all $t \in [0,1]$ we have
\[f(\gamma(lt)) \le (1-t)f(x) + t f(x') + \eps t(1-t)d(x,x').\]
Furthermore, we say that $f$ is approximately convex if it is approximately convex at every $z \in X$.

Note that if balls in $X$ are convex and for all $\lambda > 0$, the restriction of $f$ to a ball centered at $z$ if $\lambda$-convex, then $f$ is approximately convex at $z$.

\begin{proposition}\label{prop-ac-uag}
Let $X$ be a geodesic space, $f : X \to (-\infty,\infty]$ be a function that is approximately convex at $z \in X$, and $\alpha \in \R$ such that $\alpha \ge f(z)$. Then $\epi f$ is $\UAG$ at $(z,\alpha)$.
\end{proposition}
\begin{proof}
Let $\eps > 0$. Then we find $r > 0$ given by approximate convexity at $z$. Take $(x,\lambda),(x', \lambda') \in B((z,\alpha),r) \cap \epi f$ with $(x,\lambda) \ne (x',\lambda')$. Note that $x,x' \in B(z,r)$. Let $l_1 = d(x,x')$, $l_2 = |\lambda - \lambda'|$, $c_1 : [0,l_1] \to X$ be a geodesic from $x$ to $x'$, and define $c_2 : [0,l_2] \to \R$,  $c_2(s) = (1 - s/l_2)\lambda + (s/l_2)\lambda'$. Take $l = \sqrt{l_1^2 + l_2^2}$, and define $g : [0,l] \to \epi f$ by
\[g(t) = \left(c_1(tl_1/l), c_2(tl_2/l) + \eps \frac{t}{l}\left(1- \frac{t}{l}\right)d(x,x')\right).\]
The function $g$ is indeed well-defined since for all $t \in [0,l]$,
\begin{align*}
f(c_1(tl_1/l)) & \le \left(1-\frac{t}{l}\right)f(x) + \frac{t}{l}f(x') +  \eps \frac{t}{l}\left(1- \frac{t}{l}\right)d(x,x')\\
& \le \left(1-\frac{t}{l}\right)\lambda + \frac{t}{l}\lambda' +  \eps \frac{t}{l}\left(1- \frac{t}{l}\right)d(x,x') = c_2(tl_2/l) + \eps \frac{t}{l}\left(1- \frac{t}{l}\right)d(x,x').
\end{align*}
Let now the geodesic $\gamma : [0,l] \to X \times \R$ be given by $\gamma(t)=  \left(c_1(tl_1/l), c_2(tl_2/l)\right)$. Then for all $t \in (0,l]$,
\[\frac{d_2(g(t),\gamma(t))}{t} = \eps\left(1- \frac{t}{l}\right) \frac{d(x,x')}{l} < \eps.\]
\end{proof}

In the analysis of nonconvex nondifferentiable problems, the class of lower-$C^1$ functions plays an important role; see \cite{SP81}. A locally Lipschitz function $f:U \to \R$, where $U$ is an open subset of $\R^n$, is called {\it lower-$C^1$} if for every $x_0\in U$ there exist a neighborhood $V$ of $x_0$, a compact set $S$, and a jointly continuous function $g: V\times S\to \R$ such that for all $x\in V$ we have $ f(x) =\max_{x\in S} g(x,s) $ and the derivative of $g$ with respect to $x$ (exists and) is jointly continuous.
In \cite[Corollary 3]{DaGe04} it is established that a locally Lipschitz function  is approximately convex if and only if it is lower-$C^1$.  This fact and Proposition \ref{prop-ac-uag} directly yield the following result.

\begin{corollary}
Let $U \subseteq \R^n$ be open, $f : U \to \R$ be a lower-$C^1$ function,  $z \in U$, and $\alpha \in \R$ such that $\alpha \ge f(z)$. Then $\epi f$ is $\UAG$ at $(z,\alpha)$.
\end{corollary}
\begin{remark}
A natural example   of a lower-$C^1$ function is the sum of a convex function and a $C^1$ function.  
\end{remark} 

The simplest, most classical, and most important type of amenable sets (see \cite{RoWe98}) are simply defined by $C^1$ inequality constraints satisfying the Mangasarian--Fromovitz condition. In the following proposition we prove that these sets are $\UAG$.   

\begin{proposition}
Let $G \subseteq \R^n$ and $z \in G$. Suppose that there exists a neighborhood $U$ of $z$ and the $C^1$ functions $g_1, \ldots, g_m : U \to \R$  with the property that
\[U \cap G = \{x \in U \mid g_j(x) \le 0 \text{ for all } j \in \{1, \ldots, m\}\}.\]
Moreover, assume that there exists $d \in \R^n$ such that $Dg_j(z)(d) < 0$ for all $j \in \{1, \ldots, m\}$. Then $G$ is $\UAG$ at $z$.
\end{proposition}
\begin{proof}
Note first that we may assume that $Dg_j(z)(d) \le -2$ for all $j \in \{1, \ldots, m\}$ (otherwise scale the vector $d$ appropriately). Let $\eps \in (0,1)$ and $\eps' = \eps/\|d\|$. Since the functions $g_j$, where $j \in \{1, \ldots, m\}$, are $C^1$, they are approximately convex at $z$, so there exists $r > 0$ such that $B(z,2r) \subseteq U$ and, for any $x,y \in B(z,r)$,
\[g_j((1-s)x + sy) \le (1-s)g_j(x) + sg_j(y) + \eps' s (1-s) \|x-y\|\] 
for all $s \in [0,1]$ and all $j \in \{1, \ldots, m\}$. Moreover, again by $C^1$-smoothness, we can also suppose that $r$ is sufficiently small so that
\begin{align*}
g_j((1-s)x + sy + \eps' s (1-s) \|x-y\|d) - g_j((1-s)x + sy) &\le Dg_j(z)(\eps' s (1-s) \|x-y\|d) \\
& \quad + \eps' s (1-s) \|x-y\|
\end{align*}
for all $s \in [0,1]$ and all $j \in \{1, \ldots, m\}$. (Note that  $(1-s)x + sy + \eps' s (1-s) \|x-y\|d \in B(z,2r)$.)

Adding the above inequalities and taking into account that for $x,y \in B(z,r) \cap G$ we have $g_j(x) \le 0$ and $g_j(y) \le 0$, we get
\[g_j((1-s)x + sy + \eps' s (1-s) \|x-y\|d) \le 0\]
for all $s \in [0,1]$ and all $j \in \{1, \ldots, m\}$. 

Let $x,y \in B(z,r) \cap G$ with $x \ne y$, $l=\|x-y\|$, and define $f : [0,l] \to G$ by 
\[f(t) = \left(1-\frac{t}{l}\right)x + \frac{t}{l}y + \eps' \frac{t}{l}\left(1-\frac{t}{l}\right) \|x-y\|d.\]
Then, considering the geodesic $\gamma : [0,l] \to \R^n$, $\gamma(t) =  (1-t/l)x + (t/l)y$, we have
\[\frac{\|f(t)-\gamma(t)\|}{t} = \eps'\|d\|\left(1- \frac{t}{l}\right) < \eps\]
for all $t \in (0,l]$.

\end{proof}

In \cite{Fed59}, the definition of sets of positive reach was introduced in relation with the measure of global curvature properties without assuming differentiability. The main used tool is the behavior of the metric projection onto the set.  A closed subset $A$ of a metric space $X$ is said to have {\it positive reach} if there exists $\delta >0$ such that $P_A(x)$ is a singleton for every $x\in X$ with $\text{dist}(x, A) <\delta$. 
In \cite[Theorem 1.3]{Lyt05}, it is shown that if $A$ is a compact subset of a smooth Riemannian manifold (at least $C^3$ with a $C^2$  Riemannian tensor), then $A$ is $2$-convex if and only if $A$ has positive reach. Another concept related to positive reach is the one of proximal smoothness studied in \cite{ClaSteWol05}.

The fact that the metric projection is a singleton also gives a characterization of prox-regularity in $\R^n$. This notion was introduced for a better understanding of how local properties of $\text{dist}(\cdot, A)$ correspond to those of $P_A$. Although the definition of prox-regularity is based on the normal cone, we will employ its description  via the metric projection.  Namely, a closed subset $A$ of $\R^n$ is {\it prox-regular} at $z \in A$ if there exists $r > 0$ such that $P_A(x)$ is a singleton for all $x \in B(z,r)$ (see \cite{PolRocThi00}). Sets that are prox-regular at all their points are also called sets with unique footpoints or EFP-sets (see \cite{Kle81, Ban82}).

\begin{remark}
Let $A \subseteq \R^n$ be compact. If $A$ is prox-regular at every $z \in A$, then $A$ has finite extrinsic curvature at every $z \in A$. 

Indeed, let $r : A \to [0,\infty]$ be the function that assigns to $z \in A$ the supremum $r(z)$ of all numbers $r > 0$ such that, for all points in $B(z,r)$, $P_A$ is a singleton. Since $A$ is prox-regular at every $z \in A$, we have that $r(z) > 0$ for all $z \in A$. By \cite [Remark 4.2]{Fed59}, $r$ is continuous on $A$, so it attains its minimum in $A$. Said differently, $A$ has positive reach. Using \cite[Theorem 1.3]{Lyt05}, $A$ is $2$-convex, and the conclusion follows.
\end{remark}

In order to state a local version of the  fact pointed out in the previous remark we first prove the following property.

\begin{lemma} \label{lemma-prox-reg-pos-reach}
Let $A \subseteq \R^n$ be closed. Then $A$ is prox-regular at $z \in A$ if and only if there exists $r > 0$ such that $A \cap \overline{B}(z,r)$ has positive reach.
\end{lemma}

\begin{proof}

It is clear from the definition that if $z \in A$ and there exists $r > 0$ such that $A \cap \overline{B}(z,r)$ has positive reach, then $A$ is prox-regular at $z$.

Suppose now that  $A$ is prox-regular at $z \in A$. Then  there exists $R > 0$ such that $P_A(x)$ is a singleton for all $x \in B(z,R)$. Take $r = R/4$. Let $y \in \R^n$ with $\delta = \text{dist}(y, A \cap \overline{B}(z,r)) < r$. We show by contradiction that $P_{A \cap \overline{B}(z,r)}(y)$ is a singleton.

Suppose there exist two distinct points $p, q \in P_{A \cap \overline{B}(z,r)}(y)$.  Denote $m = \frac{1}{2}p + \frac{1}{2}q$. Then
\begin{equation}\label{lemmaprox-reg-pos-reach-eq1}
\|m - z\|^2 = \frac{1}{2}\|p - z\|^2 + \frac{1}{2}\|q - z\|^2 - \frac{1}{4}\|p - q\|^2 \le r^2 - \frac{1}{4}\|p - q\|^2.
\end{equation}
Since $\|y - m\| < \delta$, $m \notin A$. Let $u = P_A(m)$. We show that $u \in \overline{B}(z,r)$. Extend the segment $[u,m]$ beyond $m$ to the point $x$ such that $\|u - x\| = 2r$. In other words, $x = u + \frac{2r}{\|m - u\|}(m-u)$. Then $\|x - z\| \le \|x - m\| + \|m - z\| < 2r + r = 3r$, so $P_A(x)$ is a singleton. In addition, $P_A(x) = u$. To see this, denote 
\[\tau = \sup\{t \ge 0 : P_A(u + t(m-u)) = u\}.\] 
Then $\tau > 0$ as, for $t = 1$, $P_A(m) = u$. If $\tau = \infty$, the whole ray starting at $u$ in the direction $m-u$ projects onto $u$. If $\tau < \infty$,  applying \cite[Theorem 4.8.(6)]{Fed59}, we get $u + \tau(m-u) \notin B(z,4r)$. Hence,
\begin{align*}
4r & \le \|u + \tau(m-u) - z\| \le \tau\|m-u\| + \|u-z\| \le \tau \|m-u\| + \|u-m\| + \|m-z\| \\
& < \tau \|m-u\| + \|m - p\| + r \le \tau \|m-u\| + 2r,
\end{align*}
from where $\tau > 2r/\|m - u\|$. This shows that $P_A(x) = u$.

One of the angles $\angle_{m}(p,x)$ or $\angle_m(q,x)$ is at most $\pi/2$, and we can suppose that this is the case for $\angle_{m}(p,x)$. We have
\begin{align*}
\|x - u\|^2 &\le \|x - p\|^2 \le \|x - m\|^2 + \|p - m\|^2 =  (\|x - u\| - \|m - u\|)^2 + \frac{1}{4}\|p - q\|^2\\
& = \|x - u\|^2 - 2\|x - u\| \|m - u\| + \|m - u\|^2 + \frac{1}{4}\|p - q\|^2\\
& \le \|x - u\|^2 - 2\|x - u\| \|m - u\| + \|m - p\|^2 + \frac{1}{4}\|p - q\|^2\\
& = \|x - u\|^2 - 4r\|m - u\| + \frac{1}{2}\|p - q\|^2,
\end{align*}
and we get
\[\|m - u\| \le \frac{1}{8r} \|p - q\|^2.\]
Thus, using \eqref{lemmaprox-reg-pos-reach-eq1},
\begin{align*}
\|u - z\| \le \|u - m\| + \|m - z\| \le \frac{1}{8r} \|p - q\|^2 + \sqrt{ r^2 - \frac{1}{4}\|p - q\|^2} < r.
\end{align*}
This means that $u \in \overline{B}(z,r)$. Now,
\[\|y - u\| \le \|y - m\| + \|m - u\| \le \sqrt{\delta^2 - \frac{1}{4}\|p - q\|^2} + \frac{1}{8r}\|p - q\|^2 < \sqrt{\delta^2 - \frac{1}{4}\|p - q\|^2} + \frac{1}{8\delta}\|p - q\|^2 < \delta,\]
a contradiction. We conclude that $P_{A \cap \overline{B}(z,r)}$ is a singleton and, finally, that $A \cap \overline{B}(z,r)$ has positive reach.
\end{proof}

Combining Lemma \ref{lemma-prox-reg-pos-reach}, \cite[Theorem 1.3]{Lyt05}, and Remark \ref{rmk-extc-2-convex}, we get the following equivalences.

\begin{proposition}\label{prop-prox-reg-ext-curv-loc}
Let $A \subseteq \R^n$ be closed and $z \in A$. Then the following are equivalent:
\begin{enumerate}
\item[(i)] $A$ is prox-regular at $z$;
\item[(ii)] there exists $r > 0$ such that $A \cap \overline{B}(z,r)$ has positive reach;
\item[(iii)] there exists $r > 0$ such that $A \cap \overline{B}(z,r)$ is $2$-convex;
\item[(iv)] $A$ has finite extrinsic curvature at $z$.
\end{enumerate}
\end{proposition}

Another immediate consequence of Propositions \ref{prop-prox-reg-ext-curv-loc}  and \ref{prop-subset-extc} is the next result.

\begin{corollary} \label{cor-C11-Prop-P}
Let $A \subseteq \R^n$ be closed. If $A$ is prox-regular at $z \in A$, then $A$ is $\UAG$ at $z$.
\end{corollary}

\section{Alternating projections}\label{sect-alt-proj}

In this section we first introduce in a geodesic setting the two main geometric ingredients which we use in the study of the local linear convergence of alternating projections and focus on situations in which two intersecting closed sets satisfy these assumptions. Then we prove the convergence result.

\subsection{Super-regularity and separable intersection} 

\begin{definition}
Let $(X,d)$ be a geodesic space and $A \subseteq X$. We say that $A$ is {\it super-regular} at $z \in A$ if given any $\eps > 0$ there exists $r > 0$ such that any two points in $B(z,r)$ are joined by a unique geodesic segment and for all $y \in B(z,r/2)\setminus A$, $x \in P_A(y)$, and all $x' \in A \cap B(z,r)$ with $x' \ne x$, $\angle_{x}(y,x') \ge \pi/2-\eps$.
\end{definition}

\begin{remark}\label{rmk-def-superreg}
In the above definition, $d(y,x) \le d(y,z) < r/2$, which implies that $d(z,x) \le d(z,y) + d(y,x) < r$; hence $x,y,x' \in B(z,r)$.
\end{remark}

We provide below an example of a subset of $\R^2$ which shows that super-regularity is not persistent nearby. At this point we only justify why the set is not super-regular on balls centered at the origin. Its super-regularity at the origin will be discussed later on (see Example \ref{ex-tent-fct-revisited}).

\begin{example}\label{ex-tent-fct}
Define the function $g : [0,1/2] \to \R$ in the following way: if $t \in (1/2^{n+1},1/2^n]$, where $n \in \N$ with $n \ge 1$, let
$$
g(t)=\left\{
\begin{array}{ll}
\ds \frac{1}{2^n}\left(t - \frac{1}{2^{n+1}}\right) & \mbox{if } t \in \ds \left(\frac{1}{2^{n+1}}, \frac{3}{2^{n+2}}\right],\\[4mm]
\ds \frac{1}{2^n}\left(\frac{1}{2^{n}} - t\right) & \mbox{if } t \in \ds \left(\frac{3}{2^{n+2}}, \frac{1}{2^n}\right].
\end{array}\right.
$$
In addition, take $g(0) = 0$.

Let $A = \left\{(t,g(t)) \mid 0 \le t \le 1/2 \right\}$. Then there is no $R > 0$ such that $A$ is super-regular at all points in $B((0,0), R) \cap A$.

To see this, observe first that for every $R > 0$ there exists a sufficiently large $n$ so that 
\[z = \left(\frac{3}{2^{n+2}}, \frac{1}{2^{2n+2}}\right) \in B((0,0), R) \cap A.\]
We show that $A$ is not super-regular at $z$. Let $\eps = \arctan(1/2^{n+1})$ and $r > 0$. Take 
\[t \in \left(\frac{1}{2^{n+1}}, \frac{3}{2^{n+2}}\right) \quad \text{and} \quad t' = \frac{3}{2^{n+1}} - t\]
such that, denoting $x = (t,g(t))$ and $x' = (t',g(t'))$, $\|z - x\| = \|z-x'\| < r/4$. Let $y \in \R^2 \setminus A$ be in the interior of the triangle with vertices $z, (0, 1/2^{n+1}), (0,1/2^n)$ such that $x \in P_A(y)$ and $d(y,x) < r/4$. Then $d(y,z) < r/2$ and
\[\angle_{x}(y,x') = \frac{\pi}{2} - \angle_{x}(z,x') = \frac{\pi}{2} - \arctan(1/2^n) < \frac{\pi}{2} - \eps.\]
\end{example}

\begin{remark}
As amenability, prox-regularity, and the existence of a finite extrinsic curvature are persistent nearby and imply super-regularity (according to \cite[Proposition 4.8]{LewLukMal09}, Corollary \ref{cor-C11-Prop-P}, Proposition \ref{prop-subset-extc}, and Theorem \ref{thm-prop-p-superreg}), Example \ref{ex-tent-fct}  also shows that $A$ does not satisfy these properties at $(0,0)$.
\end{remark}

In what follows we prove that uniform approximation by geodesics implies super-regularity in the context of Alexandrov spaces with an upper curvature bound. 

\begin{theorem} \label{thm-prop-p-superreg}
Let $\kappa > 0$, $X$ be a $\CAT(\kappa)$ space, and $A \subseteq X$. If $A$ is $\UAG$ at $z \in A$, then $A$ is super-regular at $z$.
\end{theorem}

\begin{proof}
Fix $\eps \in (0,1)$. For
\begin{equation}\label{thm-superreg-subset-eq1}
\eps' = \frac{1-\cos(\eps/2)}{2} \in (0,1/2),
\end{equation}
we find $R > 0$ given by uniform approximation by geodesics. Take $r = \min\left\{R, D_\kappa/8\right\}$. Then each two points in $B(z,r)$ are joined by a unique geodesic segment. Let $y \in B(z, r/2) \setminus A$ and  $x \in P_{A}(y)$. By Remark \ref{rmk-def-superreg}, $x \in B(z,r)$. Let $x' \in  B(z,r) \cap A$ with $x' \ne x$. Then there exist $f : [0,l] \to A$ with $f(0) = x$ and $f(l) = x'$ and a geodesic $\gamma : [0,l] \to X$ starting at $x$ such that
\[\frac{d(\gamma(t),f(t))}{t} < \eps' \quad \text{for all } t \in (0,l].\]
Denote $v=\gamma(l)$. Then
\[l = d(v,x) \le d(v,x') + d(x',x) < \eps' l + 2r < l/2 + D_\kappa/4,\]
so $l < D_\kappa/2$.

\begin{cl}  \label{thm-superreg-subset-claim1}
$\eta = \angle_{x}\left(v,x'\right) < \eps/2$.
\end{cl}
\begin{proof}[Proof of Claim \ref{thm-superreg-subset-claim1}]

Suppose that $\eta \ge \eps/2$. By the cosine law in $M_\kappa^2$ we have
\begin{align*}
\cos(\sqrt{\kappa}\, d(v,x')) \le  \cos(\sqrt{\kappa}\, d(v,x))\cos(\sqrt{\kappa}\, d(x',x)) + \sin(\sqrt{\kappa}\, d(v,x))\sin(\sqrt{\kappa}\, d(x',x))\cos \eta.
\end{align*}
Denoting $D = d(x',x)$, we get
\begin{align*}
\cos(\sqrt{\kappa}\, \eps' l) &\le  \cos(\sqrt{\kappa}\, l)\cos(\sqrt{\kappa}\,D) + \sin(\sqrt{\kappa}\, l)\sin(\sqrt{\kappa}\, D)\cos(\eps/2)\\
& =  \cos(\eps/2)\cos(\sqrt{\kappa}\, (l - D))+ (1-\cos(\eps/2)) \cos(\sqrt{\kappa}\, l)\cos(\sqrt{\kappa}\,D)\\
& \le  \cos(\eps/2) + (1-\cos(\eps/2))\cos(\sqrt{\kappa}\, l).
\end{align*}
Note that $1 - a^2/2 \le \cos a \le 1 - a^2/2 + a^4/24$ for all $a \ge 0$. Applying the first inequality for $a = \sqrt{\kappa}\, \eps' l$ and the second one for $a = \sqrt{\kappa}\, l$, we obtain
\begin{align*}
1 - \frac{\kappa (\eps')^2l^2}{2} & \le \cos(\sqrt{\kappa}\, \eps' l) \le  \cos(\eps/2) + (1-\cos(\eps/2))\left(1 - \frac{\kappa l^2}{2} + \frac{\kappa^2 l^4}{24}\right)\\
& \le 1 - (1 - \cos(\eps/2))\frac{\kappa l^2}{2}\left(1 -\frac{\kappa l^2}{12} \right),
\end{align*}
from where
\[\eps ' > (\eps')^2 \ge (1-\cos(\eps/2))(1 - kl^2/12) >  (1-\cos(\eps/2))/2,\]
which contradicts \eqref{thm-superreg-subset-eq1}.
\end{proof}

\begin{cl} \label{thm-superreg-subset-claim2}
$\angle_{x}\left(y,x'\right) \ge \pi/2 - \eps$.
\end{cl}
\begin{proof}[Proof of Claim \ref{thm-superreg-subset-claim2}]
Suppose that $\angle_{x}\left(y,x'\right) < \pi/2 - \eps$. Denote $\xi = \angle_{x}\left(y,v\right)$. By Claim \ref{thm-superreg-subset-claim1},
\[\xi \le \angle_{x}\left(y,x'\right) + \angle_{x}\left(v,x'\right) <  \frac{\pi}{2} - \eps + \frac{\eps}{2} = \frac{\pi - \eps}{2}.\]
Then we find two points $y' \in [x,y]$ with $y' \ne x$ and $v' \in [x,v]$ with $v' \ne x$ such that in a comparison triangle $\Delta(\overline{x},\overline{y}',\overline{v}') \subseteq M_\kappa^2$ of the geodesic triangle $\Delta(x,y',v')$, we have $\overline{\xi}' = \angle_{\overline{x}}\left(\overline{y}',\overline{v}'\right) < (\pi - \eps)/2$. Thus, by Lemma \ref{lemma-small-angle}, in the triangle $\Delta(\overline{x},\overline{y}',\overline{v}')$  there exists $\overline{u} \in [\overline{x},\overline{v}']$ such that $d_{M_\kappa^2}(\overline{y}',\overline{u}) < d_{M_\kappa^2}(\overline{y}',\overline{x})$ and $0 < d_{M_\kappa^2}(\overline{x},\overline{u}) < \alpha$, where
\begin{equation}\label{thm-superreg-subset-eq4}
\alpha =  \frac{1}{6\sqrt{\kappa}} \cos \frac{\pi-\eps}{2}\sin(\sqrt{\kappa} \, d_{M_\kappa^2}(\overline{y}',\overline{x})).
\end{equation} 
Take now $u \in [x,v']$ such that $d(x,u) = d_{M_\kappa^2}(\overline{x},\overline{u})$; i.e., $\overline{u}$ is the comparison point of $u$. It follows that $d(y',u) \le d_{M_\kappa^2}(\overline{y}',\overline{u}) < d_{M_\kappa^2}(\overline{y}',\overline{x}) = d(y',x)$. Denote $t = d(x,u) \le l$. Then $t < \alpha < d(y',x)/6$, $u = \gamma(t)$, and $d(u,f(t))  < \eps' t$. Moreover,
\begin{align}\label{thm-superreg-subset-eq6}
\begin{split}
&\cos(\sqrt{\kappa}\, d(y',u)) \ge \cos (\sqrt{\kappa} \, d_{M_\kappa^2}(\overline{y}',\overline{u})) \\
& \quad = \cos (\sqrt{\kappa} \, t)\cos (\sqrt{\kappa} \, d_{M_\kappa^2}(\overline{y}',\overline{x}))+ \sin (\sqrt{\kappa} \, t)\sin (\sqrt{\kappa} \, d_{M_\kappa^2}(\overline{y}',\overline{x})) \cos \overline{\xi}'\\
& \quad \ge \cos (\sqrt{\kappa} \, t)\cos (\sqrt{\kappa} \, d(y',x)) + \sin (\sqrt{\kappa} \, t)\sin (\sqrt{\kappa} \, d(y',x)) \cos \frac{\pi - \eps}{2}.
\end{split}
\end{align}
At the same time, since $x \in P_{A}(y')$, we have 
\[d(y',x) \le d(y',f(t)) \le d(y',u) + d(u,f(t)) < d(y',u) + \eps' t,\]
so
\begin{align*}
\cos (\sqrt{\kappa} \, d(y',x)) & \ge \cos(\sqrt{\kappa}(d(y',u) + \eps' t))\\
& = \cos (\sqrt{\kappa} \, d(y',u)) \cos  (\sqrt{\kappa} \,  \eps' t)  - \sin (\sqrt{\kappa} \, d(y',u)) \sin  (\sqrt{\kappa} \,  \eps' t)\\
& \ge \cos (\sqrt{\kappa} \, d(y',u)) \cos  (\sqrt{\kappa} \, \eps'  t)  - \sin (\sqrt{\kappa} \, d(y',x)) \sin  (\sqrt{\kappa} \, \eps'  t).
\end{align*}
Consequently,
\begin{align*}
&  \cos (\sqrt{\kappa} \, d(y',x)) \\
& \quad \ge \left( \cos (\sqrt{\kappa} \,  t)\cos (\sqrt{\kappa} \, d(y',x)) + \sin (\sqrt{\kappa} \, t)\sin (\sqrt{\kappa} \, d(y',x)) \cos \frac{\pi - \eps}{2}\right)\cos  (\sqrt{\kappa} \, \eps' t)\\
& \quad \quad \quad  - \sin (\sqrt{\kappa} \, d(y',x)) \sin  (\sqrt{\kappa} \, \eps'  t)  \quad \text{by } \eqref{thm-superreg-subset-eq6}\\
& \quad = \cos (\sqrt{\kappa} (d(y',x) - t)) \cos (\sqrt{\kappa} \, \eps'  t)  - \sin (\sqrt{\kappa} \, d(y',x)) \sin  (\sqrt{\kappa} \, \eps'  t)\\
& \quad \quad \quad - \left(1-\cos \frac{\pi - \eps}{2}\right) \sin (\sqrt{\kappa} \, t)\sin (\sqrt{\kappa} \, d(y',x)) \cos  (\sqrt{\kappa} \, \eps'  t)\\
& \quad \ge \cos (\sqrt{\kappa} (d(y',x) - t)) \left(1 - \frac{\kappa (\eps' t)^2}{2} \right) - \sin (\sqrt{\kappa} \, d(y',x)) \sqrt{\kappa} \, \eps' t\\
&  \quad \quad \quad - \left(1-\cos \frac{\pi - \eps}{2}\right) \sqrt{\kappa} \,  t \sin (\sqrt{\kappa} \, d(y',x)),
\end{align*}
from where
\begin{align*}
\sin (\sqrt{\kappa} \, d(y',x)) \sqrt{\kappa}\, t \left(\eps' + 1-\cos \frac{\pi - \eps}{2}\right) & \ge \cos (\sqrt{\kappa} (d(y',x) - t)) - \cos (\sqrt{\kappa} \, d(y',x))  - \frac{\kappa (\eps'  t)^2}{2} \\
& \ge 2 \sin \frac{\sqrt{\kappa} \, t}{2} \sin \left(\sqrt{\kappa} \left(d(y',x) - \frac{t}{2} \right)\right) - \kappa \, t^2.
\end{align*}
Note that, as $\eps \in (0,1)$,
\[ \eps' = \frac{1-\cos(\eps/2)}{2} < \frac{1}{2}\cos\frac{\pi-\eps}{2}.\]
Denote $a=\sqrt{\kappa} \, t/2 < \sqrt{\kappa} \, \alpha/2 < 1/12$ and $b = \sqrt{\kappa} \, d(y',x)$. We have $\sin a \ge a - a^3/6 \ge a - 2a^2$ and
\begin{align*}
\sin(b-a) & = \sin b \cos a - \cos b \sin a = \sin b - \sin b (1 - \cos a) - \cos b \sin a\\
& \ge \sin b - (1 - \cos a + \sin a) \ge \sin b - (a^2/2 + a) \ge \sin b -2a > 0.
\end{align*}
Hence,
\begin{align*}
\sin (\sqrt{\kappa} \, d(y',x)) \sqrt{\kappa}\,  t \left(1-\frac{1}{2}\cos \frac{\pi - \eps}{2}\right) & > \left(\sqrt{\kappa} \, t - \kappa\, t^2 \right) \left(\sin(\sqrt{\kappa}d(y',x)) - \sqrt{\kappa}\, t \right) - \kappa \, t^2\\
& > \sqrt{\kappa} \, t \sin(\sqrt{\kappa}\,d(y',x))  - 3\kappa\, t^2.
\end{align*}
After dividing by $3 \kappa\, t$, rearranging, and using \eqref{thm-superreg-subset-eq4}, we get $t > \alpha$, a contradiction.
\end{proof}
This proves that $A$ is super-regular at $z$.
\end{proof}
 
\begin{example}\label{ex-tent-fct-revisited}
The set $A$ defined in Example \ref{ex-tent-fct} is $\UAG$ at $(0,0)$; hence it is also super-regular there.

To show this, let $\eps > 0$, and take $n \in \N$ sufficiently large so that $\eps > 1/2^n$. Let $x,x' \in B((0,0), 1/2^n) \cap A$. Suppose $x = (s,g(s))$ and $x' = (s',g(s'))$, where $s,s' \in [0,1/2^n]$. 

If $s < s'$, define $f : [0,s'-s] \to A$ by $f(t) = (s+t, g(s+t))$ and the geodesic $\gamma : [0,s'-s] \to \R^2$ starting at $x$ by $\gamma(t) = (s+t, g(s))$. Then $f(0)=x$, $f(s'-s) =x'$, and an exercise shows that $\|\gamma(t) - f(t)\| \le t/2^n < \eps t$ for all $t \in (0,s'-s]$.

Similarly, if $s' < s$, define $f : [0,s-s'] \to A$, $f(t) = (s-t, g(s-t))$, and the geodesic $\gamma : [0,s-s'] \to \R^2$, $\gamma(t) = (s-t, g(s))$. 
\end{example}

\begin{remark} \label{rmk-tent-fct}
Examples \ref{ex-tent-fct} and \ref{ex-tent-fct-revisited} allow us to conclude that uniform approximation by geodesics is not persistent nearby (otherwise $A$ would be super-regular on a ball centered at $(0,0)$, which is not the case). 
\end{remark}

We define and study next the separable intersection property for two sets.

\begin{definition} 
Let $(X,d)$ be a geodesic space and $A,B \subseteq X$. We say that $A$ {\it intersects} $B$ {\it separably} at $z \in A \cap B$ if there exist $\alpha, r > 0$ such that any two points in $B(z,r)$ are joined by a unique geodesic segment and for all $x \in (A \cap B(z,r)) \setminus B$, $y \in P_B(x) \setminus A$ satisfying $\max\{d(y,x),d(y,z)\}<r/2$ and all $x' \in P_A(y)$, $\angle_y(x,x') \ge \alpha$.
\end{definition}

\begin{remark}
Observe that in the above definition, $x,y,x' \in B(z,r)$ because $d(y,x') \le d(y,x) < r/2$, so $d(z,x') \le d(z,y) + d(y, x') < r$. 
\end{remark}

Although we are primarily interested in the nonconvex case, to get a clearer picture about the separable intersection property in a nonlinear setting, we will also consider the scenario where both sets are convex. We show first that, under additional appropriate regularity conditions imposed on the geodesic space, the separable intersection property holds if both sets are convex and one of them contains an open ball that intersects the other set. To this end we consider the following notion which, for two sets that are prox-regular at a point, can be regarded as an analogue at that point of the notion of transversality from Euclidean spaces, where it implies separable intersection.

\begin{definition}
Let $X$ be a geodesic space and $A, B \subseteq X$. We say that $A$ and $B$ are {\it transversal} at $z \in A \cap B$ if there is no nonconstant geodesic $\gamma : [0,l] \to X$ such that $z=\gamma(l/2)$ and $z \in P_A(\gamma(0)) \cap P_B(\gamma(l))$ (and hence $z \in P_A(\gamma(t)) \cap P_B(\gamma(l-t))$ for all $t \in [0,l/2]$).
\end{definition}

In the setting of $\CAT(\kappa)$ spaces, transversality at a given point $z$ still implies separable intersection at $z$ assuming local compactness and an extension property for geodesics joining points that are close to $z$.

\begin{proposition}\label{prop-trans-sep}
Let $\kappa > 0$, $X$ be a complete and locally compact $\CAT(\kappa)$ space, and $A,B$ be closed convex subsets of $X$ that are transversal at $z \in A \cap B$. Suppose also that $X$ has the $c$-geodesic extension property at $z$, where $c \in (0,D_\kappa/4)$. Then $A$ intersects $B$ separably (and, by symmetry, $B$ intersects $A$ separably) at $z$.
\end{proposition}
\begin{proof}
Observe first that two points in $B(z,2c)$ are joined by a unique geodesic segment. We argue by contradiction. Suppose that $A$ does not intersect $B$ separably at $z$. Then for every $n \ge 2$ there exist $x_n \in (A \cap B(z,2c/n)) \setminus B$, $y_n \in P_B(x_n) \setminus A$ satisfying $\max\{d(x_n,y_n),d(z,y_n)\} < c/n$, and $x'_n \in P_A(y_n)$ such that $\alpha_n=\angle_{y_n}(x_n,x'_n) < 1/n$. 

Both sequences $(x_n)$ and $(y_n)$ converge to $z$. Moreover, $x_n,y_n,x'_n \in B(z,c)$ for all $n \ge 2$, and we can consider the geodesics $\gamma_A^n,\gamma_B^n : [0,c] \to X$ satisfying $\gamma_A^n(0)=x'_n$, $\gamma_A^n(t_1)=y_n$ for some $t_1 \in (0,c)$ and $\gamma_B^n(0)=y_n$, $\gamma_B^n(t_2)=x_n$ for some $t_2 \in (0,c)$. Notice that the geodesic segments determined by $\gamma^n_A$ and $\gamma^n_B$ are contained in $B(z,2c)$. Denoting $u_n=\gamma^n_A(c)$ and $v_n=\gamma^n_B(c)$, we have $\angle_{y_n}(u_n,v_n) > \pi - 1/n$ and $d(y_n,u_n) =  c - d(x'_n,y_n) \to c$ as $d(x'_n,y_n) \le d(x_n,y_n) < c/n$. Now the $\CAT(\kappa)$ condition yields $d(u_n,v_n) \to 2c$.

By local compactness, we may assume that $(u_n)$ and $(v_n)$ are convergent to some $u$ and $v$, respectively (otherwise consider convergent subsequences). Because $d(z,u) = c, d(z,v)=c$ and $d(u,v) = 2c$, if $\gamma:[0,2c]\to X$ is the unique geodesic from $u$ to $v$, then $z=\gamma(c)$. 

We show next that $z = P_A(u)$. Let $w \in A$, $w \ne z$. If $d(z,w) > 2c$, then $d(u,w) \ge d(z,w) - d(u,z) > c = d(u,z)$. Suppose next that $d(z,w) \le 2c$. As $x'_n = P_A(y_n)$, $\angle_{x'_n}(u_n,w) \ge \pi/2$, and so, by the upper semi-continuity of the Alexandrov angle, $\angle_z(u,w) \ge \pi/2$. This shows that $d(u,w) > d(u,z)$. Similarly, $z=P_B(v)$, which means that $A$ and $B$ are not transversal at $z$, and the proof is now complete.
\end{proof}

The next two rather technical lemmas are intermediate steps in proving the transversality of convex sets for which the aforementioned intersection condition holds.

\begin{lemma}\label{lemma-trans-proj}
Let $\kappa > 0$, $X$ be a complete $\CAT(\kappa)$ space, and $A,B$ be closed convex subsets of $X$. If there is a point $z \in  A \cap B$ such that $A$ and $B$ are not transversal at $z$, then there exists a nonconstant geodesic $\gamma : [0,l] \to X$ such that $\gamma(l/2) = z$, $P_{\gamma([0,l/2])}(u) = z$ for all $u \in A \cap B(z,D_\kappa/2)$ and $P_{\gamma([l/2,l])}(v) = z$ for all $v \in B \cap B(z,D_\kappa/2)$. Consequently, $P_{\gamma([0,l])}(w) = z$ for all $w \in (A \cap B) \cap B(z,D_\kappa/2)$.
\end{lemma}
\begin{proof}
The failure of transversality of $A$ and $B$ at $z$ yields the existence of a nonconstant geodesic $\gamma:[0,l] \to X$, where $l < D_\kappa/2$, such that $\gamma(l/2)=z$ and $\{z\} = P_A(\gamma(t))\cap P_B(\gamma(1-t))$ for all $t \in [0,l/2]$. Let $u \in A \cap B(z,D_\kappa/2)$ and $t \in [0,l/2]$. Then $\angle_z(u,\gamma(t)) \ge \pi/2$ because $z=P_A(\gamma(t))$ which shows that $d(u,\gamma(t)) \ge d(u,z)$, and hence $z=P_{\gamma([0,l/2])}(u)$. A similar argument yields $P_{\gamma([l/2,l])}(v) = z$ for all $v \in B \cap B(z,D_\kappa/2)$.
\end{proof}

\begin{lemma}\label{lemma-proj-ball}
Let $\kappa' \le \kappa$, $X$ be a complete $\Re_{\kappa',\kappa}$ domain and $z \in X$. Suppose that $X$ has the $c$-geodesic extension property at $z$, where $c \in (0,D_\kappa/4)$. If $\gamma:[0,l]\to X$ is a nonconstant geodesic with $\gamma(l/2)=z$, then for any $w \in B(z,c)$ with $P_{\gamma([0,l])}(w)=z$ there is no $\delta > 0$ such that $P_{\gamma([0,l/2])}(u) = z$ for all $u \in B(w,\delta)$.
\end{lemma}
\begin{proof}
Suppose, on the contrary, that there exist such a point $w$ and a positive number $\delta$. Notice first that any two points in $B(z,2c)$ are joined by a unique geodesic segment. In addition, $w \ne z$.

As $X$ is a $\CAT(\kappa)$ space and $P_{\gamma([0,l])}(w)=z$, we have $\angle_z(\gamma(0),w) \ge \pi/2$ and $\angle_z(\gamma(l),w) \ge \pi/2$. Recall that in any $\CBB(\kappa')$ space, the sum of adjacent angles is $\pi$, so $\angle_z(\gamma(0),w) = \angle_z(\gamma(l),w) = \pi/2$. 

Extend the geodesic from $w$ to $z$ beyond $z$ to a geodesic of length $c$, and let $z'$ be its new endpoint. Then $z' \notin \gamma([0,l])$ (to see this, note, e.g., that otherwise one of the angles $\angle_z(\gamma(0),w)$ or $\angle_z(\gamma(l),w)$ would measure $\pi$).

Choose a sequence $(z_n) \subseteq \gamma([0,l/2])$ such that $z_n \to z$ and $0 < d(z,z_n) < c$ for all $n \in \N$.  Denote $\alpha_n = \angle_{z'}(z,z_n)$. By the $\CAT(\kappa)$ condition, $\alpha_n \to 0$.

Extend the geodesic from $z'$ to $z_n$ beyond $z_n$ to a geodesic of length $c$, and denote the new endpoint by $u_n$. Because $\alpha_n \to 0$, the $\CBB(\kappa')$ condition yields $d(w,u_n) \to 0$, so $d(w,u_{n_0}) < \delta$ for some $n_0 \in \N$. This means that $P_{\gamma([0,l/2])}(u_{n_0}) = z$; hence $d(u_{n_0},z) < d(u_{n_0},z_{n_0})$.

Using again the fact that the sum of adjacent angles is $\pi$, we get $\angle_z(z_{n_0},z') = \pi/2$; thus, by the $\CAT(\kappa)$ condition, $d(z',z_{n_0}) \ge d(z',z)$. 

Consequently,
\[d(z',u_{n_0}) = d(z',z_{n_0}) + d(z_{n_0},u_{n_0}) > d(z',z) + d(z,u_{n_0}),\]
a contradiction.
\end{proof}

\begin{theorem} \label{thm-int-transv}
Let $\kappa' \le \kappa$, $X$ be a complete $\Re_{\kappa',\kappa}$ domain, and $A,B$ be closed convex subsets of $X$. Suppose that $X$ has the $c$-geodesic extension property at $z \in A \cap B$, where $c \in (0,D_\kappa/4)$. If there exists $w \in \inte(A) \cap B$ with  $d(w,z) < c$, then $A$ and $B$ are transversal at $z$.
\end{theorem}
\begin{proof}
Let $\delta \in (0,D_\kappa/4)$ so that $B(w,\delta) \subseteq A$. Note that $B(w,\delta) \subseteq B(z,D_\kappa/2)$. The result follows by applying Lemmas \ref{lemma-trans-proj} and \ref{lemma-proj-ball}.
\end{proof}

As an immediate consequence of the above result and Proposition \ref{prop-trans-sep}, in the presence of local compactness, we can obtain the  separable intersection property.

\begin{theorem}
Under the assumptions of Theorem \ref{thm-int-transv}, if $X$ is additionally locally compact, then $A$ intersects $B$ separably and $B$ intersects $A$ separably at $z$.
\end{theorem}

\begin{remark} Actually we only need the $\CAT(\kappa)$ and $\CBB(\kappa')$ conditions to hold for sufficiently small triangles that are close to the point where we aim to show that the sets intersect separably.
\end{remark}

We finish the study of the separable intersection property with a result for two possibly nonconvex sets. Namely, we show that if two super-regular subsets of an Alexandrov space with an upper curvature bound intersect at a nonzero angle (in the sense assumed below), then they intersect separably.

\begin{theorem} \label{thm-angle-sep-int}
Let $\kappa > 0$, $X$ be a $\CAT(\kappa)$ space, and $A,B$ be two subsets of $X$ that are super-regular at $z \in A \cap B$. If there exist $\sigma, R > 0$ such that $\angle_{z}(p,q) \ge \sigma$ for all $p \in (B(z,R) \cap A) \setminus B$ and all $q \in (B(z,R) \cap B) \setminus A$, then $A$ intersects $B$ separably (and, by symmetry, $B$ intersects $A$ separably) at $z$.
\end{theorem}

\begin{proof}
We can suppose that $R < D_\kappa/2$ is small enough so that the sum of the angles of a triangle in $B(z,R)$ is at most $\pi + \sigma/4$.

Using the super-regularity of $A$ and $B$ at $z$ we find $0 < r < R$ such that if $x \in (B(z,r) \cap A) \setminus B$, $y \in P_{B}(x) \setminus A$ with $\max\{d(x,y),d(y,z)\} < r/2$ and $x' \in P_{A}(y)$, then $\angle_y(x,z) \ge \pi/2 - \sigma/4$ and if $x' \notin B$, $\angle_{x'}(y,z)\ge \pi/2 - \sigma/4$.

Observe that if $x' \in B$, then, by super-regularity of $B$ and the fact that $\sigma \le \pi$,
\[\angle_y(x,x') \ge \frac{\pi}{2} - \frac{\sigma}{4} \ge \frac{\sigma}{4}.\] 
Thus, we suppose next that $x' \notin B$. Then
\[\angle_y(x',z) \le \pi + \frac{\sigma}{4} - \angle_{x'}(z,y) - \angle_{z}(y,x') \le \pi + \frac{\sigma}{4} - \frac{\pi}{2} + \frac{\sigma}{4} - \sigma = \frac{\pi}{2} - \frac{\sigma}{2}.\]
Hence,
\[\angle_y(x,x') \ge \angle_y(x,z) - \angle_y(x',z) \ge \frac{\pi}{2} - \frac{\sigma}{4} - \frac{\pi}{2} +  \frac{\sigma}{2} = \frac{\sigma}{4}.\]
\end{proof}

\subsection{Local linear convergence of alternating projections} 

Let $(X,d)$ be a metric space and $A,B \subseteq X$ two nonempty sets. An {\it alternating projection sequence} $(x_n) \subseteq X$ starting at $x_0 \in A$ satisfies the conditions
\[x_{2n+1} \in P_B(x_{2n}) \quad \text{and} \quad  x_{2n+2} \in P_A(x_{2n+1})\]
for all $n \in \mathbb{N}$.

We recall that a sequence $(x_n)$ converges linearly to a point $x$ if there exists a positive constant $k$ and a rate $a \in (0,1)$ such that $d(x_n,x) \le k\cdot a^n$ for all $n \in \N$.
\begin{theorem}\label{thm-main-conv}
Let $\kappa > 0$, $X$ be a complete $\CAT(\kappa)$ space, and $A, B$ be closed subsets of $X$. Suppose that, at $z \in A \cap B$, $A$ is super-regular (or in particular $\UAG$) and intersects $B$ separably. Then any alternating projection sequence $(x_n)$ starting at $x_0 \in A$ sufficiently close to $z$ converges linearly to a point in $A \cap B$.
\end{theorem}
\begin{proof}
Let $\alpha \in (0,\pi/2)$ and $r \in (0,D_\kappa/2)$ be given by the separable intersection property. Take $\eps \in (0,1)$ sufficiently small such that
\[c'=\frac{\cos \alpha + \sin \eps}{1 - \sin \eps} < 1.\]
We can assume that $r$ is small enough so that, on the one hand, it satisfies the super-regularity condition for $\eps$ and, on the other hand, 
\[c= \frac{c'}{\cos^2(r\sqrt{\kappa}/2)} < 1.\]

Consider the starting point $x_0 \in A \cap B(z,(1-c)r/4)$, and denote $D = d(z,x_0)$.

\begin{claim}
If for some $n \in \N$, $d(z,x_{2n+1})< r/2$ and $d(x_{2n},x_{2n+1})< r/2$, then 
\[d(x_{2n+1},x_{2n+2}) \le c d(x_{2n},x_{2n+1}).\]
\end{claim}
\begin{proof}[Proof of Claim]
Denote for simplicity $x=x_{2n}$, $y=x_{2n+1}$, and $x'=x_{2n+2}$. Note first that if $x=y$, then $x'=y$, and in this case the desired inequality is obvious. Thus we can assume that $x \notin B$ and $y \notin A$. Moreover, since $d(z,y) < r/2$ and $d(x,y) < r/2$, it follows that $d(z,x) < r$. 

Take $\Delta(\overline{x},\overline{y},\overline{x}')$, a comparison triangle in $M^2_\kappa$, for the geodesic triangle $\Delta(x,y,x')$. Because $B$ intersects $A$ separably at $z$, $\angle_{y}(x,x')\ge \alpha$. This shows, in particular, that $x \ne x'$. In addition, $\angle_{\overline{y}}(\overline{x},\overline{x}') \ge \alpha$ and so, applying the cosine law in $M^2_\kappa$ we obtain
\[\cos(\sqrt{\kappa}\, d(x,x'))  \le \cos(\sqrt{\kappa}\, d(x',y))\cos(\sqrt{\kappa}\, d(x,y))+ \sin(\sqrt{\kappa}\, d(x',y))\sin(\sqrt{\kappa}\, d(x,y))\cos \alpha.\]
At the same time, by super-regularity of $A$ at $z$, $\angle_{x'}(x,y) \ge \pi/2 - \eps$, hence $\angle_{\overline{x}'}(\overline{x},\overline{y}) \ge \pi/2 - \eps$. Again by the cosine law in $M^2_\kappa$ we have
\[\cos(\sqrt{\kappa}\, d(x,y)) \le \cos(\sqrt{\kappa}\, d(y,x'))\cos(\sqrt{\kappa}\, d(x,x'))+ \sin(\sqrt{\kappa}\, d(y,x'))\sin(\sqrt{\kappa}\, d(x,x'))\sin \eps.\]
Consequently,
\begin{align*}
\cos(\sqrt{\kappa}\, d(x,y)) & \le  \cos^2(\sqrt{\kappa}\, d(y,x'))\cos(\sqrt{\kappa}\, d(x,y))\\
& \quad + \cos(\sqrt{\kappa}\, d(y,x'))\sin(\sqrt{\kappa}\, d(x',y))\sin(\sqrt{\kappa}\, d(x,y))\cos \alpha\\
& \quad + \sin(\sqrt{\kappa}\, d(y,x'))\sin(\sqrt{\kappa}\, d(x,x'))\sin \eps.
\end{align*}
This yields
\begin{align*}
& \sin(\sqrt{\kappa}\, d(y,x'))\cos(\sqrt{\kappa}\, d(x,y)) \\
& \quad \le \cos(\sqrt{\kappa}\, d(y,x'))\sin(\sqrt{\kappa}\, d(x,y))\cos \alpha  + \sin(\sqrt{\kappa}\, d(x,x'))\sin \eps\\
& \quad \le \cos(\sqrt{\kappa}\, d(y,x'))\sin(\sqrt{\kappa}\, d(x,y))\cos \alpha + \sin(\sqrt{\kappa}(d(x,y) + d(y,x')))\sin \eps.
\end{align*}
Here we also used the fact that $d(x,x') \le d(x,y) + d(y,x') \le 2d(x,y) < r < D_\kappa/2$. Thus,
\begin{align*}
(1-\sin \eps)\sin(\sqrt{\kappa}\, d(y,x'))\cos(\sqrt{\kappa}\, d(x,y)) \le (\cos\alpha + \sin \eps)\cos(\sqrt{\kappa}\, d(y,x')))\sin(\sqrt{\kappa}\, d(x,y)),
\end{align*}
from where 
\begin{equation}\label{thm-main-conv-eq1}
\tan(\sqrt{\kappa}\, d(x',y)) \le c' \tan(\sqrt{\kappa}\, d(x,y)).
\end{equation}

Let $a \in (0,\pi/2)$, and define the function $f:[0,a] \to \R$, $f(t) = \cos^2 a \tan t - t$. One can easily see that $f(t) \le 0$ for all $t \in [0,a]$. Taking $a=r\sqrt{\kappa}/2$ and $t=\sqrt{\kappa}\, d(x,y) \le r\sqrt{\kappa}/2$, we get $c' \tan(\sqrt{\kappa\, }d(x,y)) \le c\sqrt{\kappa}\, d(x,y)$. This together with \eqref{thm-main-conv-eq1} and the fact that $\sqrt{\kappa}\, d(x',y) \le \tan(\sqrt{\kappa}\, d(x',y))$ proves the claim.
\end{proof}

We show next by induction that 
\[d(z,x_{2n+1})  \le 2D\frac{1-c^{n+1}}{1-c}, \quad d(x_{2n},x_{2n+1}) \le Dc^n, \quad \text{and} \quad d(x_{2n+1},x_{2n+2}) \le Dc^{n+1}.\]
For $n=0$, the verification is straightforward:
\[d(x_0,x_1) \le d(x_0,z) = D, \quad d(z,x_1) \le d(z,x_0) + d(x_0,x_1) \le 2D,\]
and, since $D < r/4$, applying the Claim,
\[d(x_1,x_2) \le cd(x_0,x_1) \le Dc.\]
Suppose now that the inequalities hold for $n=k$. We prove that they also hold for $n=k+1$. To see this, note that
\[d(x_{2k+2},x_{2k+3}) \le d(x_{2k+2},x_{2k+1}) \le Dc^{k+1} < \frac{r}{2}\]
and
\begin{align*}
d(z,x_{2k+3}) &\le d(z,x_{2k+1}) + d(x_{2k+1},x_{2k+2}) + d(x_{2k+2},x_{2k+3})\\
& \le 2D\frac{1-c^{k+1}}{1-c} + Dc^{k+1} + Dc^{k+1} = 2D\frac{1-c^{k+2}}{1-c} < \frac{r}{2}.
\end{align*}
We can now apply the Claim to get $d(x_{2k+3},x_{2k+4}) \le cd(x_{2k+2},x_{2k+3}) \le Dc^{k+2}$, which finishes the induction proof.

The sequence $(x_n)$ is Cauchy because
\[\sum_{i \ge 2m+1}d(x_i,x_{i+1}) \le 2D\sum_{i \ge m+1}c^i = 2D\frac{c^{m+1}}{1-c} \]
and
\[\sum_{i \ge 2m}d(x_i,x_{i+1}) \le Dc^m + 2D\frac{c^{m+1}}{1-c} = Dc^m\frac{1+c}{1-c}\]
for all $m \in \N$, so
\[d(x_n,x_{n+k}) \le D(\sqrt{c})^n\frac{1+c}{1-c}\] 
for all $n,k \in \N$. This means that $(x_n)$ converges to some $z' \in A \cap B$ and the convergence is linear with rate $\sqrt{c}$ since
\[d(x_n,z') \le D(\sqrt{c})^n\frac{1+c}{1-c}\]
for all $n \in \N$. Note also that by an appropriate choice of $\eps$ and $r$, the rate can be made arbitrarily close to $\sqrt{\cos \alpha}$.
\end{proof}

We hope that the theoretical ideas developed here help illuminate the fundamentals of alternating projection methods.  Their practical application is a topic of ongoing research;  here, we confine ourselves to one very simple nonconvex example as a concrete illustration of our main ideas.

The setting that we consider is the $n$-dimensional spherical space $\mathbb{S}^n$. Recall that the $n$-dimensional sphere is the set
\[\mathbb{S}^n = \{x \in \R^{n+1} \mid \langle x, x \rangle = 1\},\]
where $\langle \cdot, \cdot \rangle$ is the Euclidean scalar product. Endowed with the distance $d: {\mathbb{S}}^n \times {\mathbb{S}}^n \to \mathbb{R}$ that assigns to each $(x,y) \in {\mathbb{S}}^n \times {\mathbb{S}}^n$ the unique number $d_{\mathbb{S}^n}(x,y) \in [0,\pi]$ such that $\cos d_{\mathbb{S}^n}(x,y) = \langle x, y \rangle$, ${\mathbb{S}}^n$ is a geodesic space. In other words, $d_{\mathbb{S}^n}(x,y)$ is the length of the smallest arc of a great circle in ${\mathbb{S}}^n$ which joins $x$ and $y$. 

\begin{example}

Take a unit vector $a$ and the set $A = \{x \in \mathbb{S}^n \mid \langle x, a \rangle = 1/2\}$. This set is nonconvex but super-regular at all its points. 

To see this, let $z \in A$ and $\eps \in (0,1)$. Take $r = \arccos(\cos^2(\eps))/2$, and let $y \in B(z,r/2) \setminus A$, $x = P_A(y)$, and $x' \in A \cap B(z,r)$ with $x' \ne x$. Then $d_{\mathbb{S}^n}(a,x) = d_{\mathbb{S}^n}(a,x') = \pi/3$. Note also that $y \in [a,x]$. Denoting $D = d_{\mathbb{S}^n}(x,x')$, we have
\[D \le d_{\mathbb{S}^n}(x,z) + d_{\mathbb{S}^n}(z,x') \le 2r = \arccos(\cos^2(\eps)).\] 
Applying the spherical cosine law in the spherical  triangle with vertices $a,x,x'$, we get
\begin{align*}
\frac{1}{2} = \cos d_{\mathbb{S}^n}(a,x') & = \cos d_{\mathbb{S}^n}(a,x) \cos D +  \sin d_{\mathbb{S}^n}(a,x) \sin D \cos \angle_x(y,x') \\
& = \frac{1}{2} \cos D +  \frac{\sqrt{3}}{2} \sin D \cos \angle_x(a,x'),
\end{align*}
so
\[\cos \angle_x(a,x') = \frac{1}{\sqrt{3}}\cdot\frac{1 - \cos D}{\sin D} < \sqrt{1 - \cos D} < \sin \eps = \cos (\pi/2 - \eps).\]
This shows that $\angle_x(y,x') = \angle_x(a,x') > \pi/2 - \eps$.

Take another unit vector $b$ such that $\langle a, b \rangle \in (-\sqrt{3}/2, 0) \cup (0,\sqrt{3}/2)$, and let $B$ be the subset of $\mathbb{S}^n$ orthogonal to $b$. The sets $A$ and $B$ will intersect, and we fix $z \in A \cap B$. Note that $B$ is super-regular at $z$ since it is weakly convex. Moreover, $A$ and $B$ intersect separably there (by Theorem \ref{thm-angle-sep-int}), and an easy exercise shows that the alternating projection sequence for the sets $A$ and $B$ starting at a point $x \in A$ consists of the following update:

\begin{algorithm}[H]
\SetAlgoLined
  \While{$x \notin B$}{
   $\ds x' = \frac{x - \langle b , x \rangle b}{\|x - \langle b , x \rangle b\|};$\\[1mm]
   $\ds x'' = \frac{\sqrt{3}}{2}\cdot\frac{x' - \langle a , x' \rangle a}{\|x' - \langle a , x' \rangle a\|} + \frac{1}{2}a;$\\[1mm]
   $x=x''$\;
 }
\end{algorithm}

By Theorem \ref{thm-main-conv}, starting near $z$, the alternating projection sequence converges linearly to a point in $A \cap B$.
\end{example}

\section{Acknowledgements}
We would like to express our gratitude to Alexander Lytchak for his interest and detailed comments which significantly improved Section \ref{sect-curv-conv}. We also wish to thank the referees for their constructive remarks.

Research supported in part by National Science Foundation Grant DMS-2006990, by DGES Grant PGC2018-098474-B-C21, and by a grant of the Ministry of Research, Innovation and Digitization, CNCS/CCCDI -- UEFISCDI, project number PN-III-P1-1.1-TE-2019-1306, within PNCDI III.

\end{document}